\documentclass[12pt,a4paper,final]{amsproc}
\usepackage[english]{babel}
\usepackage{latexsym}
\usepackage{amsmath}
\usepackage[all]{xy}
\usepackage{amssymb}
\usepackage{amscd}
\usepackage[cp850]{inputenc}
\usepackage{mathrsfs}
\usepackage[notcite,notref]{showkeys}
\theoremstyle{plain}

\newtheorem{theorem}{Theorem}[section]
\newtheorem{proposition}[theorem]{Proposition}
\newtheorem{cor}[theorem]{Corollary}
\newtheorem{lemma}[theorem]{Lemma}
\theoremstyle{remark}

\theoremstyle{remark}
\newtheorem{defin}[theorem]{Definition}

\newcommand{\R}{\mathbb{R}}
\newcommand{\N}{\mathbb{N}}
\newcommand{\K}{\mathbb{K}}

\newcommand{\adef}{\begin{defn}}
\newcommand{\zdef}{\end{defn}}
\newtheorem{defn}[theorem]{Definition}

\def\Hom{\operatorname{Hom}}
\def\Ext{\operatorname{Ext}}

\def\supp{\operatorname{supp}}
\def\PB{\operatorname{PB}}
\def\PO{\operatorname{PO}}

\def\Ext{\operatorname{Ext}}

\def\Hom{\operatorname{Hom}}
\newcommand{\aproof}{\begin{proof}}
\newcommand{\zproof}{\end{proof}}

\newcommand{\To}{\longrightarrow}

\title{On the ${\Ext}^2$-problem for Hilbert spaces}

\author[Cabello]{F\'elix Cabello S\'anchez}
\address{(FC, JC, RG) Instituto de Matem\'aticas Imuex\\ Universidad de Extremadura\\
Avenida de Elvas\\ 06071-Badajoz\\ Spain} \email{fcabello@unex.es\\
castillo@unex.es, rgarcia@unex.es}
\author[Castillo]{Jes\'{u}s M.F. Castillo}

\author[Corr\^ea]{Willian H. G. Corr\^ea}
\address{(WC, VF) Departamento de Matem\'atica, Instituto de Matem\'atica e
Estat\'\i stica, Universidade de S\~ao Paulo, rua do Mat\~ao 1010,
05508-090 S\~ao Paulo SP, Brazil} \email{willhans@ime.usp.br, ferenczi@ime.usp.br}

\author[Ferenczi]{\\Valentin Ferenczi}
\address{(VF) Equipe d'Analyse Fonctionnelle \\
Institut de Math\'ematiques de Jussieu \\
Universit\'e Pierre et Marie Curie - Paris 6 \\
Case 247, 4 place Jussieu \\
75252 Paris Cedex 05 \\
France.} 

\author[Garc\'ia]{Ricardo Garc\'ia}

\thanks{The research of the first, second and fifth authors was supported by Project MINCIN MTM2016-76958-C2-1-P and Project IB16056 de la Junta de Extremadura. The third author was supported by S\~{a}o Paulo Research Foundation (FAPESP), processes 2016/25574-8 and 2018/03765-1. The fourth author was
supported by CNPq, grant 303034/2015-7 and Fapesp, grant
2016/25574-8.}

\subjclass[2010]{46B25,46M18}

\begin{document}



\maketitle
\begin{abstract} We show that $\Ext^2(\ell_2, \ell_2)\neq 0$ in the category of Banach spaces. This solves a sharpened version of Palamodov's problem
and provides a solution to the second order version of Palais problem. We also show that $\Ext^2(\ell_1, \K)\neq 0$ in the category of quasi Banach spaces, which solves the four-space problem for local convexity. \end{abstract}

\section{Introduction}
This paper is devoted to solving the main problem left open in \cite{cck} and \cite{castrica}: Is $\Ext^2(\ell_2, \ell_2)=0$? To answer the
question means to decide whether every four-term exact sequence of Banach spaces and operators
$$\begin{CD} 0 @>>> \ell_2 @>>> E @>>> F @>>> \ell_2 @>>> 0\end{CD}$$
is trivial, that is, equivalent to zero in $\Ext^2(\ell_2, \ell_2)$. The difficulties to do this are of two kinds:
\begin{itemize}
\item[(a)] How to construct such an object.
\item[(b)] How to decide if it is trivial or not.
\end{itemize}
What makes (a) so difficult is the perfect homogeneity of Hilbert space. General methods provide examples of Banach spaces $X, Y$ so that $\Ext^2(X,Y)\neq 0$ usually appealing to some incomparability property between projective and injective presentations, or between $X$ and $Y$, or their (complemented) subspaces or quotients. In the case of $\ell_2$ nothing works: the space is self-similar in any way, as well as its subspaces, quotients or dual. Moreover, attempts to use the standard reduction of degree  technique fail because very few things are known either about $K^1(\ell_2)$,
the kernel of a quotient map $\ell_1\To \ell_2$ or its dual space $\ell_\infty/\ell_2$. In particular, no known technique allows one to obtain a nontrivial elements in either $\Ext(K^1(\ell_2), \ell_2)$ or $\Ext(\ell_2, \ell_\infty/\ell_2)$. In turn, what makes (b) difficult is that, when $n\geq 2$, the equivalence relation in $\Ext^n$  is not easy to handle, to say the least.
\smallskip

In spite of all this we will show that $\Ext^2(\ell_2,\ell_2)\neq 0$  by means
of a counterexample which is at once surprising and optimal in every sense: optimal because, first, it has the simplest possible form
\begin{equation}\label{eq:z2z2}
\xymatrixrowsep{0.5pc}
\xymatrix{0\ar[r] & \ell_2\ar[r] & \widetilde{Z}_2  \ar[rr]  \ar[dr]  &&Z_2 \ar[r] &\ell_2\ar[r]& 0 \\
& && \ell_2 \ar[ur]   }
\end{equation}
of two twisted Hilbert sequences \emph{spliced}. Second, because the short exact sequence on the right $0\To \ell_2 \To Z_2\To \ell_2\To 0$ is Kalton-Peck's celebrated example \cite{kaltpeck}, while the one on the left $0\To \ell_2 \To \widetilde{Z}_2\To \ell_2\To 0$ is a certain vector-valued form of the Kalton-Peck sequence. And third, because the construction preserves unconditionality in the sense that Diagram~\ref{eq:z2z2} naturally belongs to the category of Banach modules over $\ell_\infty$, which considerably expands its range of applications. And surprising 
because this seems to be the only occasion in which one disproves a (serious) conjecture made by Kalton \cite{confe}.\medskip

After all this propaganda, let us explain the organization of the paper and highlight the main results. Section 2 is preliminary. It contains
the necessary background about $\Ext=\Ext^1$, quasilinear maps and centralizers, as well as some ancillary material on $\Ext^2$. Nonlinear maps find their way here to allow the construction and study of the short exact sequences later spliced to get the required four-term sequence. The basic idea is that each
exact sequence $0\To Y \To Z\To X\To 0$ corresponds to a
quasilinear map $\Phi:X\To Y$, out from which the middle space $Z$ can be reconstructed.
The price we pay for this simplification is that our default setting has to be the category of quasi Banach spaces, even if we are mostly interested in Banach spaces. 

Section~\ref{sec:crit} contains the main criteria for the (non) triviality of sequences in $\Ext^2$. The equivalence relation of $\Ext^n$ is quite elusive when $n\geq 2$. However, a small symmetry miracle occurs at $n=2$
which provides a criterion, both visual from the diagrammatic point of view and operative
from the quasilinear point of view. We then develop some ad-hoc estimates which not only render the basic criterion manageable, but even suggest the right form of the counterexamples.\smallskip

Armed with these tools we tackle in Section~\ref{sec:counter} the task of proving the nontriviality of (\ref{eq:z2z2}). This is achieved by means of a
selective elimination of the symmetries of $Z_2$  and
 a combinatorial argument involving partitions.\smallskip

Section~\ref{sec:appl} contains a number of applications, ranging from operator theory to the (non) vanishing of $\Ext^2$ both in Banach modules and quasi Banach spaces, including:

\begin{itemize}
\item That $\Ext^2(X,Y)\neq 0$ for all Banach spaces $X$ and $Y$ containing $\ell_2^n$ uniformly complemented.
\item A proof that every four-term exact sequence obtained by splicing two copies of the same centralizer on $\ell_p$ is trivial in $\Ext^2(\ell_p,\ell_p)$; and that the same occurs with the Enflo-Lindenstrauss-Pisier quasilinear map \cite{elp}, the first quasilinear map appearing in Banach space theory, which is not a centralizer.
\item Some remarks on the Yoneda product $\Ext(\ell_p,\ell_p)\times \Ext(\ell_p,\ell_p)\To \Ext^2(\ell_p,\ell_p)$ for module extensions.
\item A proof that $\Ext^2(\ell_p,\ell_p)$ is not zero in the category of quasi Banach modules over $\ell_p$.
\item A proof that $\Ext^2(\ell_1,\K)$ is not zero in in the category of quasi Banach spaces. Precisely, a nontrivial exact sequence of quasi Banach spaces $0\To \K \To R \To B \To \ell_1\To 0$ where $\K$ is the ground field and $R$ is Ribe' space \cite{ribe}.
\end{itemize}

The problem addressed in this paper can be regarded as a sharpened version of Palamodov \cite[Problem 6]{pala}: is $\operatorname{Ext}^2(\cdot,E)=0$ for \emph{any} Fr\'echet space? Counterexamples were provided by Wengenroth \cite[pp. 177--178]{jwenge}, and
then in \cite{castrica}. In the first case the example has the form
$$
\xymatrixrowsep{0.5pc}
\xymatrix{0\ar[r] & U^* \ar[r] & L_\infty  \ar[rr]  \ar[dr]_{\imath^*}  && L_1 \ar[r] &U\ar[r]& 0 \\
& && \ell_2 \ar[ur]_\imath   }
$$
where $\imath$ is an isomorphic embedding,
while in the second case the example is
$$
\xymatrixrowsep{0.5pc}
\xymatrix{0\ar[r] & \ell_2\ar[r] & X  \ar[rr]  \ar[dr]  && L_1 \ar[r] &U\ar[r]& 0 \\
& && \ell_2 \ar[ur] _\imath  }
$$
where  $0\To \ell_2\To X \To \ell_2\To 0$ is a nontrivial twisted Hilbert space.
In both cases an uncontrolled space $U=L_1/\imath[\ell_2]$ appears that cannot be reduced to anything reasonable, let alone to a Hilbert space. Thus, the question of the vanishing of $\Ext^2(\ell_2, \ell_2)$ remained untractable by general methods.

The introduction of homological methods in Banach space theory was fueled by the attempt to solve what is known as Palais problem: does there exist a twisted Hilbert space that is not (isomorphic to) a Hilbert space? The existence of such object was first proved by Enflo, Lindenstrauss and Pisier \cite{elp}, but the construction which is of paramount importance to our purposes is that of the
Kalton-Peck space $Z_2$ appeared soon later in \cite{kaltpeck}.
Thus, our proof that $\Ext^2(\ell_2, \ell_2)\neq 0$ can be viewed as a solution to the ``second order'' Palais problem. 
The closely related fact that $\Ext^2(\ell_1, \K)\neq 0$, within the category of quasi Banach spaces, can be viewed as an optimal solution of the ``four-spaces problem'' for local convexity. The classical ``three-space problem'' for local convexity was solved independently and almost simultaneously by  Roberts \cite{robe}, Ribe \cite{ribe} and Kalton~\cite{kalt} in the late seventies.

\section{Preliminaries}\label{sec:pre}
Although most of the time we deal with Banach spaces which are in fact quite similar to Hilbert spaces, the natural setting for our study is the category of quasi Banach spaces and linear bounded operators, that we denote by $\mathsf Q$ when necessary. The subcategory of Banach spaces is denoted by $\mathsf{B}$. Our main examples share a certain unconditional structure that makes them into quasi Banach modules over the Banach algebra $\ell_\infty$ in a natural way and so it will also be convenient to consider the category of quasi Banach $\ell_\infty$-modules and linear bounded homomorphisms as well. We are aware that this approach may annoy some readers, but we hope they will feel more comfortable as the paper progresses. By the way of compensation we have made a great effort to present all our results avoiding interpolation theory (compare, when it corresponds, with \cite{cck}).

A presentation of the basic elements of (quasi) Banach space theory, including quasilinear maps and twisted sums, very akin to the approach of this paper, can be found in \cite{kalthb}. The article \cite{k-m} explains the connections between centralizers and interpolation theory and many other things.

\subsection{Quasi Banach spaces and modules}
A quasinorm on a linear space $X$ is a function $\|\cdot\|:X\To [0,\infty)$ satisfying the following conditions: 
\begin{itemize}
\item $\|\lambda x\|=|\lambda|\|x\|$ for every $x\in X$ and every scalar $\lambda$; $\|x\|=0$ if and only if $x=0$.
\item There is a constant $C$ such that $\|x+y\|\leq C\big(\|x\|+\|y\|\big)$ for every $x,y\in X$.
\end{itemize}
A quasinormed space is a linear space $X$ equipped with a quasinorm. Such a space carries a linear topology for which the unit ball $\{x:\|x\|\leq1\}$ is a neighbourhood of zero. If the resulting topological vector space is complete we call it a quasi Banach space. 
A quasinormed module over a Banach algebra $A$ is a quasinormed space $X$ together with a jointly continuous outer product $A\times X\To X$ satisfying the traditional algebraic requirements. A quasi Banach module is a complete quasinormed module.

If $X$ and $Y$ are quasinormed modules over $A$, an homomorphism $f:X\To Y$  is an operator such that $f(ax)=a(f(x))$ for every $x\in X$ and $a\in A$.

The only Banach algebra that we need in this paper is the algebra $\ell_\infty$ of bounded functions $a:\N\To\K$ endowed with the sup norm.

\subsection{Homology}
We assume from the reader some acquaintance with the basic elements of homology as expounded in the classic books \cite{hiltstam, lane} or in the functional analysis-oriented notes \cite{ecfa}.

Let us however fix the notation by recalling a few definitions.
A (finite or infinite) sequence of quasi Banach spaces and operators
$$
\xymatrixcolsep{3pc}
\xymatrix{
\cdots \ar[r] & E_{n-1}\ar[r]^{u_{n-1}}& E_{n}\ar[r]^{u_n} & E_{n+1}\ar[r] &\cdots
}
$$
is  exact when the kernel of each operator agrees with the image of the preceding one:  $u_{n-1}[E_{n-1}]=\ker u_n$ for every $n$ under consideration.

Let us first consider short exact sequences, that is, exact sequences of the form
\begin{equation}\label{eq:YAX}
\xymatrix{
0 \ar[r] & Y \ar[r]^\imath & E \ar[r]^\pi & X\ar[r] & 0
}
\end{equation}
The reason behind this notation is that one treats the end spaces $X$ and $Y$ as ``fixed'' while the middle space $E$ is considered ``variable''. In this setting the sequence $\xymatrix{
0 \ar[r] & Y \ar[r] & F \ar[r] & X\ar[r] & 0
}
$
is equivalent to (\ref{eq:YAX}) if there is an operator $u$ making the following diagram commutative:
$$
\xymatrix{
0 \ar[r] & Y \ar[r] \ar@{=}[d] & E \ar[r]\ar[d]^u & X\ar[r]\ar@{=}[d] & 0\\
0 \ar[r] & Y \ar[r] & F \ar[r] & X\ar[r] & 0
}
$$
This is really an equivalence relation since a fortunate combination of algebra (the five-lemma, see \cite[Lemma 1.1]{hiltstam} or \cite[Corollary 3.23]{ecfa}) and  topology (the open mapping theorem) guarantees that $u$ is a linear homeomorphism. 
Then
 $\Ext(X,Y)$ is defined as the set of short exact sequences of the form (\ref{eq:YAX}) modulo equivalence. This set can be given a natural (= functorial) linear structure via pull-back and push-out constructions, called Baer' sum, that the reader can see in \cite[Chapter IV, \S~9]{hiltstam} or \cite[Chapter 6]{ecfa}. The zero element of  $\Ext(X,Y)$ is the (class of the) trivial sequence $\xymatrix{
0 \ar[r] & Y \ar[r] & Y\oplus X \ar[r] & X\ar[r] & 0
}
$ with the obvious operators. It turns out that (\ref{eq:YAX}) is trivial if and only if it splits, in the sense that there is an operator $P:E\To Y$ such that $P\,\imath= {\bf I}_Y$ or, equivalently, there is an operator $J: X\To E$ such that $\pi\, J={\bf I}_X$.
\smallskip

The definition of $\Ext^2(X,Y)$ starts the same. 
Given two exact sequences $0\To Y\To E_1\To E_2\To X\To 0$ and $0\To Y\To F_1\To F_2\To X\To 0$, denoted by $\mathscr E$ and $\mathscr F$, respectively, we write
$\mathscr E\to \mathscr F$ or $\mathscr F\leftarrow \mathscr E$ if there
is a commutative diagram
$$
\xymatrix{
0 \ar[r] & Y \ar[r] \ar@{=}[d] & E_1 \ar[r]\ar[d]^{u_1} &  E_2 \ar[r]\ar[d]^{u_2} & X\ar[r]\ar@{=}[d] & 0\\
0 \ar[r] & Y \ar[r] & F_1 \ar[r] &  F_2 \ar[r] & X\ar[r] & 0
}
$$
The maps $u_1$ and $u_2$ are not longer isomorphisms and we obtain just a partial order for four-term sequences. Nevertheless, this partial order generates an equivalence relation by declaring $\mathscr E$ and $\mathscr F$ equivalent (written $\mathscr E\sim\mathscr F$) if there is
 a finite chain $(\mathscr C_i)_{1\leq i\leq n}$ so
that
\begin{equation}\label{eq:A=B}
\mathscr E \to \mathscr C_1 \leftarrow \mathscr C_2\to \dots
\leftarrow \mathscr C_n\to \mathscr F.
\end{equation}
Although it will not be used in this paper, it can be remarked that, as can be seen in \cite[Corollary 6.40]{ecfa}, two links are enough. The space $\Ext^2(X, Y)$ is the set of equivalence classes of four-term exact sequences
$0\To Y\To E_1\To E_2\To X\To 0$. It carries a linear structure whose zero element is the class of the trivial sequence
$$\begin{CD}
0@>>> Y@>{\bf I}>> Y @>0>> X @>{\bf I}>> X @>>> 0
\end{CD}$$
Regarding the subcategories of Banach spaces and quasi-Banach $\ell_\infty$-modules, one can formalize the corresponding definitions of
$$\Ext_{\mathsf B}(X,Y),\qquad \Ext_{\mathsf B}^2(X,Y),\qquad \Ext_{\ell_\infty}(X,Y),\qquad \Ext_{\ell_\infty}^2(X,Y)$$
in the obvious way. Observe that if a given exact sequence of Banach spaces is zero in $\Ext^2_\mathsf B(X,Y)$, then so is in $\Ext^2(X,Y)$, and the same happens for quasi Banach modules. The converse seems to be unknown for Banach spaces and known to be false for modules, even for short exact sequences. The reason is that some of the links (spaces or arrows) in (\ref{eq:A=B}) might not be in the corresponding subcategory.

\subsection{Quasilinear maps}\label{sec:quasi}
The study of short exact sequences of (quasi) Banach spaces is greatly simplified
by the use of quasilinear maps. A map $\Phi: X\To Y$ acting between quasinormed spaces is quasilinear if it is homogeneous ($\Phi(\lambda x)=\lambda\Phi(x)$ for every $x\in X$ and $\lambda\in\mathbb K$) and satisfies an estimate
$$\|\Phi(x+y)- \Phi(x) - \Phi(y)\|\leq Q \big(\|x\|+\|y\|\big)$$
for some constant $Q\geq 0$ and all $x,y\in X$.
When necessary, we denote by $Q(\Phi)$ the least possible constant for which the preceding inequality holds.

 These maps make their way into the homology of (quasi) Banach spaces because of the following construction. Assume $X$ and $Y$ are quasi Banach spaces and that $\Phi: X_0\To Y$ is quasilinear, where $X_0$ is a dense subspace of $X$. Then
the functional
$$ \|(y,x)\|_\Phi=\|y-\Phi(x)\|+\|x\|$$
is a quasinorm on the product space $Y\times X_0$. If we denote by $Y\oplus_{\Phi}X_0$ the resulting quasinormed space, then we have an exact sequence
$$
\xymatrix{
0\ar[r] & Y \ar[r]^<<<<\imath & Y\oplus_{\Phi}X_0  \ar[r]^<<<<\pi & X_0\ar[r] & 0,
}
$$
where $\imath(y)=(y,0)$ and $\pi(y,x)=x$. It is clear that both $\imath$ preserves the quasinorms, while $\pi$ maps the unit ball of  $Y\oplus_{\Phi}X_0 $ onto that of $X_0$. Thus, if $Z(\Phi)$ denotes the completion of $Y\oplus_{\Phi}X_0$, then $\pi$ extends to a quotient map $Z(\Phi)\To X$ that we call again $\pi$ which yields the exact sequence of quasi Banach spaces
$$
\xymatrix{
0\ar[r] & Y \ar[r]^<<<<\imath & Z({\Phi})  \ar[r]^<<<<\pi & X\ar[r] & 0,
}
$$
called, for good reason, the sequence induced by $\Phi$. All short exact sequences of quasi Banach spaces arise in this way, up to equivalence, although we will not use this fact. We say that two quasilinear maps $\Phi,\Psi:X_0\To Y$ are strongly equivalent if they induce equivalent quasinorms on $Y\times X_0$, equivalently, if
there is a constant $K$ such that $\|\Psi(x)-\Phi(x)\|\leq K\|x\|$ for all $x\in X_0$.

This applies, in particular, when $X$ and $Y$ are Banach spaces. One however must be aware that there are short exact sequences
$
\xymatrix{
0\ar[r] & Y \ar[r] & Z \ar[r] & X\ar[r] & 0}
$
in which $X$ and $Y$ are Banach spaces but $Z$ is just a quasi Banach space. This cannot occur when the quotient space is $B$-convex, something that all the spaces $\ell_p$ are for $1<p<\infty$, and we have $\Ext(X,Y)=\Ext_{\mathsf B}(X,Y)$. These technical points will play a secondary role later.

\subsection{Centralizers}
These are a very special type of quasilinear maps that have to do with module structures. Centralizers have their own philosophy which does not fit exactly into the general framework of quasilinear maps as expounded in the preceding Section; see \cite[Section~8]{k-m}. As the only algebra that plays a role in our exposition is $\ell_\infty$, we can provide the reader with the minimal background one needs to understand the paper right now.

By a (quasinormed) sequence space we understand a linear space $X$ of functions
$x:\N\To\mathbb K$ equipped with a quasinorm $\|\cdot\|$ such that:
\begin{itemize}
\item The finitely supported functions are dense in $X$.
\item If $|y|\leq |x|$ and $x\in X$, then $y\in Y$ and $\|y\|\leq \|x\|$.
\end{itemize}
If $e_i$ is the function that takes the value $1$ at $i$ and vanishes elsewhere
we may also require that $\|e_i\|=1$ for every $i\in\N$. If $X$ is complete we call it a quasi Banach sequence space and we call it a Banach sequence space when the quasinorm is a norm. The easiest examples are the spaces $\ell_p$ for $0<p<\infty$. These are Banach spaces when $p\geq 1$. According to this,  $\ell_\infty$ itself is not a sequence space, but $c_0$ is. A sequence space is an $\ell_\infty$-module under the pointwise product: if $x\in X$ and $a\in\ell_\infty$, then $ax\in X$ and $\|ax\|\leq \|a\|_\infty\|x\|$.
From now on module structures refer to $\ell_\infty$ unless otherwise stated and we denote by $X^0$ the subspace of finitely supported sequences of $X$.

\begin{defin} A centralizer is a homogeneous mapping $\Phi: X\To Y$ acting between quasinormed modules that obeys an estimate of the form
$$
\|\Phi(ax)-a\Phi(x)\|\leq C\|a\|_\infty\|x\|,
$$
for some constant $C$ and all $a\in\ell_\infty, x\in X$.
\end{defin}

Our interest in centralizers stems from the fact that if $X$ and $Y$ are quasi Banach modules, $X_0$ is a dense submodule of $X$ and $\Phi:X_0\To Y$ is a quasilinear centralizer, then the product $a(y,x)=(ay,ax)$ makes $Y\oplus_\Phi X_0$ into a quasinormed module (and so $Z(\Phi)$ is a quasi Banach module) and the induced sequence lives in the category of quasi Banach modules. As a partial converse, if $X$ is a quasi Banach sequence space, then each short exact sequence of quasi Banach modules $0\To Y\To Z\To X\To 0$ arises, up to equivalence, from a quasilinear centralizer $\Phi:X^0\To Y$. We have the following remark  about ``automatic quasilinearity'' of centralizers.

\begin{lemma}\label{lem:autom}
Every centralizer defined on a quasinormed sequence space is quasilinear.
\end{lemma}

\subsection{Kalton-Peck maps and their chunked versions}\label{sec:chunk}
The most famous quasilinear maps were introduced by Kalton-Peck in \cite{kaltpeck} and they are actually centralizers. Let $\varphi:\R^+\To\mathbb K$ be a Lipschitz function such that $\varphi(0)=0$. Then, for $0<p<\infty$, the map $\Omega^\varphi_p:\ell_p^0\To \ell_p$ defined by
\begin{equation}\label{eq:KP}
\Omega^\varphi_p= x\,\varphi\left(\frac{p}{2}\log\frac{\|x\|}{|x|}\right)
\end{equation}
is a (quasilinear) centralizer whose constants $C\big(\Omega^\varphi_p\big)$ and  $Q\big(\Omega^\varphi_p\big)$ depend only on $p$ and the Lipschitz constant of $\varphi$.
We now introduce the second (type of) centralizer that we need to carry out the proof of the main result. As all centralizers do, it stems from complex interpolation theory (see \cite[Theorem 3.2]{new}), but we will consider a direct approach which does not require any previous knowledge on interpolation.\medskip

Let $(A_i)$ be a partition of $\N$ that we consider fixed in all what follows. Fix $0<p<\infty$.
Each $x\in\ell_p$ can be written as $x=\sum_ix_i$, where the $i$-th summand is
$x_i=x1_{A_i}$, with $\|x\|_p=\big( \sum_i \|x_i\|_p^p\big)^{1/p}$. Using this notation, one has:

\begin{lemma}\label{lem:chunk}
The map $\widetilde{\Omega}^\varphi_p: \ell_p^0\To\ell_p$ defined by
\begin{equation}\label{eq:chunk}
\widetilde{\Omega}^\varphi_p(x)=\sum_ix_i\varphi\left(\frac{p}{2}\log\frac{\|x\|}{\|x_i\|}\right)
\end{equation}
is a centralizer.
\end{lemma}

\begin{proof}
Pick $x\in\ell_p^0$ and $a\in\ell_\infty$ and set $a_i=a1_{A_i}$. Then
$$
\widetilde{\Omega}^\varphi_p(ax)=\sum_ia_ix_i\varphi\left(\frac{p}{2}\log\frac{\|ax\|}{\|a_ix_i\|}\right),\qquad
a\widetilde{\Omega}^\varphi_p(x)=\sum_ia_ix_i \varphi\left(\frac{p}{2}\log\frac{\|x\|}{\|x_i\|}\right).
$$
Now,
\begin{align*}
\|\widetilde{\Omega}^\varphi_p(ax)- a\widetilde{\Omega}^\varphi_p(x)\|_p&
= \left( \sum_i \|a_ix_i\|^p\left|\varphi\Big(\frac{p}{2}\log\frac{\|ax\|}{\|a_ix_i\|}\Big)- \varphi \Big( \frac{p}{2} \log\frac{\|x\|}{\|x_i\|}\Big)\right|^p\right)^{1/p}\\
&\leq \left( \sum_i \|a_i\|_\infty^p\|x_i\|_p^p\left|\varphi\Big(\frac{p}{2}\log\frac{\|ax\|}{\|a_ix_i\|}\Big)- \varphi \Big( \frac{p}{2} \log\frac{\|x\|}{\|x_i\|}\Big)\right|^p\right)^{1/p}\\
&\leq C(\Omega^\varphi_p)\|a\|_\infty\|x\|_p,
\end{align*}
where the last inequality is just the centralizer estimate of the Kalton-Peck map $\Omega^\varphi_p$ applied to the sequences $(\|a_i\|_\infty)_i$ and $(\|x_i\|_p)_i$.
\end{proof}

The choice of the function $\varphi(t)=t$ in (\ref{eq:KP}) and (\ref{eq:chunk}) is especially rewarding and we simply write $\Omega_p$ and $\widetilde\Omega_p$ for these maps. When $p=2$ we will even omit the subscript.
We also write $Z_p(\varphi)=Z(\Omega^\varphi_p)$ and $\widetilde{Z}_p(\varphi)=Z(\widetilde\Omega^\varphi_p)$. If $\varphi$ is the identity on $\R_+$ we just write $Z_p$ and $\widetilde{Z}_p$.


\subsection{Mirrored centralizers}\label{sec:mirror}
One of the core ideas in the widely ignored memoir \cite{kaltmem} is that centralizers never walk alone. To give shape to this affirmation, let us recall a result from the even more ignored \cite[Lemma 5(a) and Corollary 3]{cabe}:

\begin{lemma}\label{lem:compa} Let $0<q<p<\infty$. If $\Phi_p$ is a centralizer on $\ell_p$, then the map defined by
$$
\Phi_q(f)=u|f|^{q/r}\Phi\big{(}|f|^{q/p}\big{)}
$$
for $q^{-1}=r^{-1}+p^{-1}$ is a centralizer on $\ell_q$. All centralizers on $\ell_q$ arise in this way, up to strong equivalence.
\end{lemma}

We will thus say that $\Phi_q$ is the reflection of $\Phi_p$ in $\ell_q$. 
As the reader may guess, there is a reason behind the slightly eccentric presentation of the centralizers of the preceding Section: parameters have been assigned in such a way that if $q<p$, then $\Omega_q^\varphi$ is the reflection of $\Omega_p^\varphi$ in $\ell_q$, and the same occurs to their chunked versions. This will be used in Sections~\ref{sec:comm} and \ref{sec:all p}.

\section{Splitting criteria}\label{sec:crit}
Suppose we are given two quasi Banach spaces $X$ and $Y$ and that we want to construct  a four-term exact sequence
$
\xymatrix{0\ar[r] & Y\ar[r] & E_1 \ar[r]   &E_2 \ar[r] &X\ar[r]& 0  }
$.
The simplest (and unique) way to achieve this is to pick another space $E$, then construct two short exact sequences
\begin{align*}
\xymatrix{ &&&& 0\ar[r] & Y\ar[r]^\imath & E_1 \ar[r]^\pi   &E \ar[r] & 0 &&&& (\mathscr E_1)}\\
\xymatrix{&&&&0\ar[r] & E\ar[r]^\jmath & E_2 \ar[r]^\varpi   &X \ar[r] & 0&&&& (\mathscr E_2)}
\end{align*}
and \emph{splice them through $E$} to get
$$
\xymatrixrowsep{0.5pc}
\xymatrix{&&&0\ar[r] & Y\ar[r] & E_1 \ar[rr]^{\jmath\,\pi}
  \ar[dr]  &&E_2 \ar[r] &X\ar[r]& 0&&& \!\!\!\!\!\!\!\!\!(\mathscr E) \\
& && &&&A \ar[ur]   }
$$
Our first task is to decide when the resulting sequence is trivial in $\Ext^2(X,Y)$.
We have the following criterion which is both visual from the diagrammatic side and
transparent from the quasilinear point of view.

\begin{lemma}\label{lem:thesquare}
Let $(\mathscr E_1)$ and $(\mathscr E_2)$
be two short exact sequences of quasi Banach spaces. The following statements are equivalent:
\smallskip

\noindent{\rm (a)}\quad The spliced sequence $(\mathscr E)$ is trivial in $\Ext^2(X,Y)$.
\smallskip

\noindent{\rm (b)}\quad The sequence  $(\mathscr E)$ splits, that is, there exists a quasi Banach space $\square$ and a commutative diagram
\begin{equation}\label{split}
\begin{CD}
&&0&&0 \\
&&@VVV @VVV\\
&&Y @= Y \\
&&@VVV @VVV \\
0@>>> E_1@>>> \square  @>>> X@>>>0\\
&&@VVV @VVV @|\\
0@>>> E @>>> E_2 @>>> X @>>>0 \\
&&@VVV @VVV \\
&&0&&0
\end{CD}
\end{equation}
with exact rows and columns.
\smallskip

\noindent Moreover, if the sequence  $(\mathscr E_1)$ is induced by the quasilinear map $\Phi: E\To Y$, then these are equivalent to:
\smallskip

\noindent{\rm (c)}\quad The map $\Phi: E\To Y$ has a quasilinear extension to $E_2$.

\smallskip

\noindent If the sequence $(\mathscr E_2)$ is induced by the quasilinear map $\Psi: X\To E$, then {\rm (a)} and {\rm (b)} are equivalent to:
\smallskip

\noindent {\rm (d)}\quad The map $\Psi: X\To E$ has a quasilinear lifting to $E_1$.
\end{lemma}

\begin{proof}
We begin with the implication  (b) $\implies$ (a).
Observe that if there is a diagram
as in (b) then there is a commutative diagram
\begin{equation*}
\xymatrixrowsep{0.5pc}\xymatrix{0\ar[r] & Y\ar[r]  \ar@{=}[dd] & Y\oplus E_1 \ar[rr] \ar[dd] \ar[dr]  &&\square \ar[dd]  \ar[r] &X\ar[r]  \ar@{=}[dd]& 0 \\
& && E_{1} \ar[ur] \ar[dd]  \\
0\ar[r] & Y\ar[r] & E_1 \ar[rr]  \ar[dr]  &&E_2 \ar[r] &X\ar[r]& 0 \\
& && E \ar[ur]  }
\end{equation*}
It follows that the lower sequence is trivial since it is equivalent to the upper one and the following diagram is commutative:
\begin{equation*}
\xymatrix{0\ar[r] & Y\ar[r]  \ar@{=}[d] & Y\oplus E_1 \ar[r] \ar[d]   &\square \ar[d]  \ar[r] &X\ar[r]  \ar@{=}[d]& 0 \\
0\ar[r] & Y \ar@{=}[r] & Y \ar[r]^0   &X \ar@{=}[r] &X\ar[r]  & 0  }
\end{equation*}

We now stablish the implication (a) $\Longrightarrow$ (b).
As the zero sequence has the property required in (b) with $\square=Y\oplus X$,
it is all a matter of showing that that property is preserved in the successive steps that define the equivalence of $\Ext^2(X,Y)$.

So, assume we have a commutative diagram
\begin{equation*}
\xymatrixrowsep{0.5pc}\xymatrix{0\ar[r] & Y\ar[r]   \ar@{=}[dd] & E_1 \ar[rr] \ar[dd]  \ar[dr]  &&E_2 \ar[r] \ar[dd] &X\ar[r] \ar@{=}[dd] & 0 \\
& && E \ar[ur] \ar[dd]  \\
0\ar[r] & Y\ar[r] & F_1 \ar[rr]  \ar[dr]  &&F_2 \ar[r] &X\ar[r]& 0 \\
& && F \ar[ur]  }
\end{equation*}
and that the sequence of $E$'s fits in a diagram as (\ref{split}). Then one can form the commutative diagram
\begin{equation}\label{cubecubecube}
\xymatrixrowsep{0.5pc}
\xymatrix{
&Y\ar@{=}[rr]\ar@{=}[dr]\ar[ddd]&& Y\ar@{=}[dr]  \ar[ddd]&&\\
&&Y\ar@{=}[rr]\ar[ddd]&&Y \ar[ddd] && \;\\
\; &&\;\\
&E_1\ar[ddd] \ar[dr]\ar[rr]&& \square \ar[ddd]\ar[rr]\ar[dr]&& X\ar@{=}[dr] \ar@{=}[ddd]\\
&&F_1\ar[rr]\ar[ddd]&&\PO \ar[rr]\ar[ddd] && X\ar@{=}[ddd]\\
\; &&\;\\
& E \ar[dr]\ar[rr]&& E_2\ar[rr]\ar[dr]&& X\ar@{=}[dr]\\
& & F \ar[rr]&& F_2 \ar[rr] &&X}\end{equation}
where the push-out refers to the inner horizontal square
\begin{equation*}
\begin{CD}
E_1@>>> \square \\
@VVV @VVV \\
F_1@>>>\PO\\
\end{CD}
\end{equation*}
and the arrows beginning at $\PO$ are defined by the universal property of the push-out contruction.

By categorical duality, that is, by ``reversing the arrows'' and taking pull-back instead of push-out when necessary, one sees that as long as one has a commutative diagram of exact sequences
\begin{equation*}
\xymatrixrowsep{0.5pc}\xymatrix{0\ar[r] & Y\ar[r]   & E_1 \ar[rr]   \ar[dr]  &&E_2 \ar[r]  &X\ar[r] & 0 \\
& && E \ar[ur]   \\
0\ar[r] & Y\ar[r]  \ar@{=}[uu] & F_1 \ar[rr]  \ar[uu] \ar[dr]  &&F_2 \ar[r]  \ar[uu]  &X\ar[r]  \ar@{=}[uu]& 0 \\
& && F \ar[ur]  \ar[uu] }
\end{equation*}
and the $E$'s satisfy condition (b), then so the $F$'s do.

\medskip

(b)$\implies$(c) This is immediate ... if one is acquainted with the connection between quasilinear maps and short exact sequences. Assume $E_1=Y\oplus_\Phi E$, where $\Phi: E\To Y$ is quasilinear and let us draw the hypothesized commutative diagram
\begin{equation}\label{dia:b->c}
  \xymatrix{Y \ar@{=}[r]\ar[d]_\imath & Y \ar[d]^I \\
    Y \oplus_\Phi E \ar[r]^J\ar[d]_\pi & \square \ar[r]\ar[d]^Q& X \ar@{=}[d] \\
    E \ar[r]^\jmath & E_2 \ar[r] & X}
    \end{equation}
    in which the 0's have been omitted. Consider the map $e\in E\longmapsto(0,e)\in Y \oplus_\Phi E$. This is a linear (probably unbounded) section of the quotient map $\pi$. Since all exact sequences of linear spaces are trivial in the purely algebraical sense this map can be ``extended'' to a linear section of $Q$. Precisely, there is a linear map $L: E_2\To \square $ such that\begin{itemize}
    \item $L(\jmath(e))=J(0,e)$ for every $e\in E$.
    \item $P\circ L$ is the identity on $E_2$.
    \end{itemize}

    In a similar vein, consider the mapping $e\in E\longmapsto(\Phi(e),e)\in Y \oplus_\Phi E$, which is a (homogeneous) bounded section of $\pi$ and ``extend'' it to a (homogeneous) bounded section of $Q$, that is, a homogeneous bounded $B: E_2\To \square $ such that
   \begin{itemize}
    \item $B(\jmath(e))=J(\Phi(e),e)$ for every $e\in E$.
    \item $B\circ L$ is the identity on $E_2$.
    \end{itemize}
It is clear that the difference $B-L$ maps $E_2$ to $\ker Q=I[Y]$ and so $\Gamma=I^{-1}\circ (B-L)$ is a quasilinear extension of $\Phi$.
\medskip

The implication (c)$\implies$(b) is obvious: if $\Gamma: E_2\To Y$ is a quasilinear map extending $\Phi$, set $\square= Y\oplus_\Gamma E_2$, put the obvious maps and check.

\medskip

(b)$\implies$(d) Assume $E_2=E\oplus_\Psi X$, where $\Psi: X\To E$ is quasilinear.
The relevant diagram is now
\begin{equation}\label{dia:b->d}
  \xymatrix{Y \ar@{=}[r]\ar[d]_\imath & Y \ar[d]^I \\
    E_1 \ar[r]^J\ar[d]_\pi & \square \ar[r]\ar[d]^Q& X \ar@{=}[d] \\
    E \ar[r]^\jmath & E\oplus_\Psi X \ar[r] & X}
    \end{equation}
Let $B:X\To\square$ be a homogeneous bounded map such that $Q(B(x))=(\Psi(x), x)$ for all $x\in X$ and $L:X\To\square$ be a linear map such that $Q(L(x))=(0, x)$ for all $x\in X$. Then $\Lambda=J^{-1}\circ(B-L)$ is the required lifting of $\Psi$, that is, $\pi\circ\Lambda=\Psi$.
\end{proof}

One disavantage of working with quasilinear maps is that, as a rule, they cannot be explicitly defined on the whole space which is aimed to be the quotient of the resulting short exact sequence, but only on a dense subspace,
and so we need to adapt our criteria to this setting.

Let $X, Y, E$ be quasi Banach spaces and let $\Psi:X_0\To E$ and $\Psi: E_0\To Y$ be quasilinear maps, where $X_0$ and $E_0$ are dense subspaces of $X$ and $E$, respectively.
Let
$$
\xymatrix{0\ar[r] & Y\ar[r]^\imath & E_1 \ar[r]^\pi   &E \ar[r] & 0},\qquad\text{and}\qquad
\xymatrix{0\ar[r] & E\ar[r]^\jmath & E_2 \ar[r]^\varpi   &X \ar[r] & 0}
$$
be the induced sequences, so that $E_1=Z(\Phi)$ and $E_2=Z(\Psi)$ are the completions of $Y\oplus_\Phi E_0$ and $E\oplus_\Psi X_0$, respectively, as explained in Section \ref{sec:quasi}.

In this setting, we say that the concatenation $\Phi\,\Psi$ is trivial in $\Ext^2(X,Y)$, and  we write $\Phi\,\Psi \sim 0$ for short, if the associated four-term sequence
$$
\xymatrixrowsep{0.5pc}
\xymatrix{0\ar[r] & Y\ar[r]^\imath & E_1 \ar[rr]  \ar[dr]_\pi  &&E_2 \ar[r]^\varpi &X\ar[r]& 0 \\
& && E \ar[ur]_\jmath   }
$$
is trivial in $\Ext^2(X,Y)$. We have the following operative version of Lemma~\ref{lem:thesquare}. The proof is a simple adaptation that we leave to the patient reader.

\begin{lemma}\label{lem:crit-operative}
Let $X, E, Y$ be quasi Banach spaces and $\Psi:X_0\To E$ and  $\Phi:E_0\To Y$ quasilinear maps such that $\Psi[X_0]\subset E_0$, where $X_0$ and $E_0$ are dense subspaces of $X$ and $A$, respectively. The following are equivalent:
\smallskip

\noindent{\rm (a)}\quad $\Phi\,\Psi\sim 0$ in $\Ext^2(X,Y)$.
\smallskip

\noindent{\rm (b)}\quad $\Phi$ has a quasilinear extension to $E_0\oplus_\Psi X_0$.

\smallskip
\noindent{\rm (c)}\quad $\Psi$ has a quasilinear lifting to $Y\oplus_\Psi E_0$.\hfill$\square$
\end{lemma}

Our next task is to render these criteria manageable; with the same notation of the preceding Lemma, one has

\begin{lemma}\label{lem:H} The following are equivalent: \medskip

\noindent{\rm (a)}\quad $\Phi\,\Psi\sim 0$ in $\Ext^2(X,Y)$.
\medskip

\noindent{\rm (b)}\quad There is a homogeneous map $H:X_0\To Y$ and a constant $K$ such that
 \[
\|H(x + x') - H(x) - H(x') - \Phi(\Psi(x + x') - \Psi(x) - \Psi(x'))\| \leq K\big{(}\|x\|+ \|x'\|\big{)}
\]
for every $x, x' \in X_0$.
\end{lemma}

	\begin{proof}(a)$\implies$(b) Let $\Gamma: E_0\oplus_\Psi X_0\To Y$ be any quasilinear extension of $\Phi$, so that $\Gamma(e,0)=\Phi(e)$ for every $e\in E_0$. Define $H(x)=\Gamma(\Psi(x),x)$ and check.\smallskip
	
	(b)$\implies$(a) Define $\Gamma: E_0\oplus_\Psi X_0\To Y$ as $\Gamma(e,x)=H(x)+\Phi(e-\Psi(x))$ and check that it is quasilinear. Obviously $\Gamma(e,0)=\Phi(e)$ for every $e\in E_0$.
\end{proof}

The function $H$ will sometimes be called \emph{a witness} that $\Psi\,\Phi \sim 0$.

\subsection*{Locally convex twisted sums and $K$-spaces}
The criteria provided by Lemma~\ref{lem:thesquare} and its relatives characterizes the triviality of spliced sequences in the category of quasi Banach spaces. The equivalence between (a) and (b) is true replacing quasi Banach spaces by Banach spaces and $\Ext^2(X,Y)$ by $\Ext^2_\mathsf{B}(X,Y)$, and the same is true for quasi Banach modules. Parts (c) and (d) are more delicate since a quasilinear map acting between Banach spaces can well lead to a quasi Banach space which is not even isomorphic to a Banach space; see \cite{ribe, kalt}. One can isolate those quasilinear maps that, acting between Banach spaces, produce Banach spaces. This approach was pursued in \cite{cabecastdu}. Those subtleties are however unnecessary in our current circumstances. Let us explain why. A minimal extension of a quasi Banach space $X$ is a short exact sequence of the form $0\To\K\To Z\To X\To 0$. If all minimal extensions of $X$ are trivial, that is, if $\Ext(X,\K)=0$, then $X$ is called a $K$-space. The spaces $\ell_p$ are $K$-spaces for all values of $0<p\leq\infty$ with the only exception of $p=1$. Moreover, if $X$ is a Banach $K$-space and one has an exact sequence $0\To Y\To Z\To X\To 0$ in which $Y$ is a Banach space, then $Z$ is isomorphic to a Banach space and so $\Ext(X,Y)=\Ext_\mathsf B(X,Y)$. The point of this discussion is to remark that when the spaces $X$ and $E$ are Banach $K$-spaces and $Y$ is any Banach space, the space $\square$ appearing in (b) is also necessarily a Banach space: just look at the left descending sequence and \emph{then} at the middle horizontal one. In this way, all the results of this Section remain true in the category of Banach spaces under the additional assumption that $X$ and $E$ are $K$-spaces.


\section{The counterexample}\label{sec:counter}

Let $X$ and $Y$ be $\ell_\infty$-modules and let $A$ be any subset of $\ell_\infty$.
We say that a function $f:X\To Y$ commutes with $A$ if $f(ax)=af(x)$ for every $x\in X$ and every $a\in A$. We are mostly interested in the
following choices of $A$:
\begin{itemize}
\item The unitary group $U=\{u:\N\To\K\, \text{ such that }\, |u|=1\}$.
\item The real unitary group $U_\R=\{u:\N\To\R \,\text{ such that }\, |u|=1\}$.
\end{itemize}
Of course $U=U_\R$ when the ground field is $\R$.
Note that if $f:X\To Y$ is a mapping acting between sequence spaces that commutes with $U_\R$, then $f$ preserves supports in the sense that $\supp (f(x))\subset \supp (x)$.
Every centralizer defined on a sequence space is strongly equivalent to one that commutes with $U$.

\begin{lemma} Let $X$ and $Z$ be quasi Banach sequence spaces and $Y$ a Banach sequence space. If $\Psi: X^0\To Z$ and $\Phi: Z^0\To Y$ are centralizers which commute with $U_\R$ and $\Phi \,\Psi \sim 0$ in $\Ext^2(X,Y)$ then one can choose a witness function $H$ that commutes with $U_\R$.
\end{lemma}

\begin{proof}
If $H:X^0\To Y$ is as in Lemma~\ref{lem:H}(b), then for every unitary $u$, and in particular for $u\in U_\R$ one has
$$
\|u^{-1}\big{(}H(u(x + x')) - H(ux) - H(ux')\big{)} - \Phi(\Psi(x + x') - \Psi(x) - \Psi(x'))\| \leq K\big{(}\|x\|+ \|x'\|\big{)}.
$$
If we identify the group $U_\R$ with the ``Cantor group'' $\{1,-1\}^\N$, it is clear that the product topology of the later corresponds to the relative weak* topology of $U_\R$ when $\ell_\infty$ is regarded as the dual of $\ell_1$. In particular $U_\R$ is a compact group for the weak* topology.  Let $du$ denote the Haar measure on $U_\R$. Observe that for finitely supported $x$, the ``orbit'' $U_\R \, x=\{ux: u\in U_\R\}$ is finite and so, the mapping
$
u\in U_\R\longmapsto u^{-1}H(ux)\in Y$
 is weak* to norm continuous.
Define a new mapping $\widetilde{H}: X^0\To Y$ by the  Bochner integral
$$
\widetilde{H}(x)=\int_{U_\R} u^{-1}H(ux) \, du
$$
and check. Note that $
\widetilde{H}(x)$ agrees with the average of $u^{-1}H(ux)$ over those real unitaries $u$ such that $u=1$ off $\supp(x)$.
\end{proof}

We now specialize to $X=Z=Y=\ell_2$ and make a careful using of the symmetries of the Kalton-Peck centralizer $\Omega$.

\begin{cor}\label{withKP} Let $\Phi: \ell^0_2\To\ell_2$ be centralizer commuting with $U_\R$. If $\,\Phi\,\Omega \sim 0$, then there exist
a homogeneous, support preserving $H:  \ell^0_2\To\ell_2$ and  a constant $K$
such that, whenever $x,y\in\ell_2^0$ are disjoint and of equal norms,
$$\|H(x+y)-H(x)-H(y) - 	\Phi(x+y) \|_2 \leq K  \|x\|_2.$$
\end{cor}

\begin{proof}
 If $x, y$ are disjointly supported and of equal norm in $\ell_2$, then
$$\Omega(x+y)-\Omega(x)-\Omega(y)=\frac{\log 2}{2} (x+y).$$
The result follows using the homogeneity of $\Phi$ and $H$.
\end{proof}

In order to exploit the estimate provided by the preceding result we need to select pairs of disjoint sums of vectors from the unit basis of $\ell_2^0$ in a judicious way. This is achieved through certain partitions, defined below.
The idea is quite simple:
\begin{itemize}
\item Fix $k\in \mathbb N$ and start with the set $2^k$ (that is, a set of cardinality $2^k$).
\item Split it into two halves with the same number of elements. Then split the resulting sets into two halves and continue until reaching the singletons.
\item Don't forget to keep track of the whole process labelling all the sets.
\end{itemize}
We can formalize this procedure by using a dyadic tree of finite height.
Let $\mathscr T_k$ be the dyadic tree of height $k$ whose elements are words of length at most $k$  written with 0s and 1s, including the ``empty word'', which has length 0.
Given $\alpha=\alpha_1\alpha_2\cdots\alpha_n$ in $\mathscr T_k$, with $n<k$, we put  $\alpha0=\alpha_1\alpha_2\cdots\alpha_n0$ and $\alpha1=\alpha_1\alpha_2\cdots\alpha_n1$.

\begin{defin}\label{def:adeq}
An adequate partition of $2^k$ is a set-valued function $I:\mathscr T_k\To\mathscr P(2^k)$ such that:
\begin{itemize}
\item $I(\varnothing)=2^k$, $I(\alpha)$ is nonempty for every $\alpha\in\mathscr T_k$.
\item If $\alpha$ has length less than $k$, then $I(\alpha)$ is the disjoint union of $I(\alpha0)$ and $I(\alpha1)$.
\end{itemize}
\end{defin}

\noindent The sets $I(\alpha)$ for $\alpha$ of fixed length $0\leq n\leq k$ form a partition on $2^k$ into $2^{n}$ many subsets of cardinality $2^{k-n}$.
It takes only a moment's reflection to realize that an adequate partition is essentially a linear order on $2^k$. Indeed for every adequate partition there exist a unique order
$\preccurlyeq$ such that $x\preccurlyeq y$ whenever  $x\in I(\alpha0)$ and $y\in I(\alpha1)$. This correspondence will be used later.
\medskip

Each adequate partition on $2^k$ gives rise to a family of vectors of $\ell_2(2^k)$, parametrized by $\alpha\in\mathscr T_k$ just taking
$$
x_\alpha= 1_{I(\alpha)}=\sum_{i\in I(\alpha)} e_i.
$$

\begin{lemma} Let $\Phi: \ell^0_2\To\ell_2$ be centralizer commuting with $U_\R$. If $\Phi\,\Omega \sim 0$ then there is a constant $K = K(\Phi)$ such that if $S$ is a subset of $\N$ with $|S|=2^k$ and $I,J:\mathscr T_k\To \mathscr P(S)$ are adequate partitions, then
\begin{equation}\label{eq:nutcracker}
\Big\|\sum_{\alpha\in\mathscr T_k}\big(\Phi(x_\alpha)-\Phi(y_\alpha) \big)	\Big\|_2\leq  2Kk\sqrt{2^{k-1}},
\end{equation}
where $x_\alpha$ and $y_\alpha$ are the vectors associated to $I$ and $J$, respectively, that is, $x_\alpha=\sum_{i\in I(\alpha)} e_i$ and $y_\alpha=\sum_{i\in J(\alpha)} e_i$ for $\alpha\in\mathscr T_k$.
\end{lemma}

\begin{proof}Let $H$ and $K$ be as in Corollary~\ref{withKP}.
Since $x_0$ and $x_1$ have the same norm and $x_0+x_1=x_\varnothing=1_S$ we have
$$\|H(x_\varnothing)-H(x_0)-H(x_1) -\Phi(x_\varnothing)\|_2 \leq K \sqrt{2^{k-1}}.$$
The same estimate holds using
$y_0$ and $y_1$. Since $y_\varnothing=x_\varnothing$ we get
\begin{equation}\label{eq:weget}\|H(x_0)+H(x_1)-H(y_0)-H(y_1)\|_2 \leq 2K\sqrt{2^{k-1}}.
\end{equation}
We claim that, for every $0<j\leq k$,
\begin{equation}\label{eq:weclaim}
\Big\|\sum_{|\alpha| = j} \big(H(x_\alpha)-H(y_\alpha) \big)
+ \sum_{|\alpha| < j}\big( \Phi(x_\alpha)-\Phi(y_\alpha)\big )\Big\|_2 \leq 2Kj\sqrt{2^{k-1}},
\end{equation}
where $|\alpha|$ denotes the length of $\alpha$. Indeed, when $j=1$ this reduces to (\ref{eq:weget}).
Now assume that (\ref{eq:weclaim}) holds for $j=i-1$ and let us check it for $j=i$.

By the induction hypothesis this comes from
\begin{align*}
\Big\|\sum_{|\alpha| = i} (Hx_\alpha&-Hy_\alpha)-
\sum_{|\gamma| = i-1}(Hx_\gamma-Hy_\gamma)
-\sum_{|\delta| < i-1}(\Phi x_\delta-\Phi y_\delta)
+ \sum_{|\alpha| < j}( \Phi x_\beta-\Phi y_\beta) \Big\|_2\\
&\leq  \Big\|\sum_{|\gamma| = i-1} H(x_{\gamma 0})+H(x_{\gamma 1})-
H(x_{\gamma 0}+x_{\gamma 1})+\Phi(x_{\gamma 0}+x_{\gamma 1})\Big\|_2\\
&+\Big\|\sum_{|\gamma| = i-1} H(y_{\gamma 0})+H(y_{\gamma 1})-
H(y_{\gamma 0}+y_{\gamma 1})+\Phi(y_{\gamma 0}+y_{\gamma 1})\Big\|_2
\end{align*}
which using that $H$ and $\Phi$ preserve supports and that
$$\| H(z_{\alpha 0})+H(z_{\alpha 1})-
H(z_{\alpha 0}+z_{\alpha 1})+\Phi(z_{\alpha 0}+z_{\alpha 1})\| \leq K\sqrt{2^{k-j}},$$
for $z=x,y$, gives an upper estimate of
$$2\sqrt{2^{j-1}} K \sqrt{2^{k-j}}=2K \sqrt{2^{k-1}}$$ and proves the claim.
When $j=k$ then all $x_\alpha$ and $y_\alpha$ are single vectors of the basis, the first summand in (\ref{eq:weclaim}) is null, and
$$\Big\|
\sum_{\alpha \in\mathscr T_k} \big(\Phi(x_\alpha)-\Phi(y_\alpha)\big)\Big\|_2
=
\Big\|
\sum_{|\alpha|< k} \big(\Phi(x_\alpha)-\Phi(y_\alpha)\big)\Big\|_2  \leq 2Kk\sqrt{2^{k-1}}.\vspace{-20pt}$$
\end{proof}
Let us say that centralizer $\Phi: X\To Y$ acting between symmetric sequence spaces is symmetric if, for every permutation $\sigma:\N\To \N$, one has $\Phi(x\circ\sigma)=(\Phi (x))\circ\sigma$. This is not the traditional definition, but every symmetric centralizer in the traditional sense is strongly equivalent to one of this form.

Note that if $\Phi$ is a symmetric centralizer, then the left-hand side of (\ref{eq:nutcracker}) is $0$; indeed in that case there there are scalars $\lambda_j$ for $0\leq j\leq k$ such that $\Phi(x_\alpha)=\lambda_{|\alpha|} x_\alpha$ and $\Phi(y_\alpha)=\lambda_{|\alpha|} y_\alpha$  for all $\alpha\in\mathscr T_k$. So, in order to obtain our counter-example we need a highly nonsymmetric centralizer.

Let us consider the following finite-dimensional versions of the centralizer $\widetilde{\Omega}$ defined in Section~\ref{sec:chunk}. Fix $m,n\in\N$ and partition the product $m\times n$ into $m$ subsets of cardinality $n$ as follows
$$
A_i=\{(i,j): 1\leq j\leq n\}\qquad(1\leq i\leq m).
$$
Then define $\widetilde{\Omega}_{m,n}: \ell_2(m\times n)\To\ell_2(m\times n)$ by
$$
\widetilde{\Omega}_{m,n}(x)=\sum_{1\leq i\leq m} x_i\log\frac{\|x\|}{\|x_i\|},
$$
where $x_i=x1_{A_i}$. By Lemma~\ref{lem:chunk} we have $C(\widetilde{\Omega}_{m,n})\leq C(\Omega)$ for all $m,n$. Note that if $x$ has exactly $q$ nonzero chunks and they all have the same norm, then
 $$
\widetilde{\Omega}_{m,n}(x)=\log(q) x.
 $$
Fix now some  $k$ and identify $2^k$ with the product $2^{r} \times 2^{s}$ with $k=r+s$. Let $\preccurlyeq$ be  the associated lexicographic order, i.e., for $(a,b), (c,d) \in 2^r \times 2^s$,
$$(a,b) \preccurlyeq (c,d) \Longleftrightarrow a < c {\rm\ or\ } (a=c {\rm \ and\ } b\leq d).$$
Let $I:\mathscr T_k\To\mathscr P(2^k)$ be the adequate partition associated to $\preccurlyeq$. We also consider the  lexicographic order ``symmetric'' to $\preccurlyeq$, namely
$$(a,b) \preccurlyeq' (c,d) \Longleftrightarrow b < d {\rm\ or\ } (b=d {\rm \ and\ } a\leq c).$$
Let $J$ be the partition induced by $\preccurlyeq'$. Now we follow the notations introduced just after Definition~\ref{def:adeq}, in particular
$$
x_\alpha=\sum_{i\in I(\alpha)}e_i,\qquad y_\alpha=\sum_{j\in J(\alpha)}e_j\qquad(\alpha\in\mathscr T_k).
$$

\begin{lemma}\label{lem:rs}
Let $k=r+s$ given, $I$ and $J$ as before, with associated vectors $(x_\alpha)$ and $(y_\alpha)$, respectively. Let ${\widetilde{\Omega}}$ be the centralizer ${\widetilde{\Omega}}_{2^r,2^s} :\ell_2(2^r\times 2^s)\To \ell_2(2^r\times 2^s)$.
Then, for every $\alpha\in\mathscr T_k$ one has
\begin{eqnarray*}
|\alpha| \geq r &\Longrightarrow&  {\widetilde{\Omega}}(x_\alpha)=0, \\
|\alpha|<r &\Longrightarrow&{\widetilde{\Omega}}(x_\alpha)=(r-|\alpha|)\log(2)\, x_\alpha,\\
|\alpha| \geq s &\Longrightarrow& {\widetilde{\Omega}}(y_\alpha)=(k-|\alpha|)\log(2) y_\alpha,\\
|\alpha|<s &\Longrightarrow& {\widetilde{\Omega}}(y_\alpha)=r \log(2)\, y_\alpha.\end{eqnarray*}
In particular
$$
\Big\|
\sum_{\alpha \in \mathscr T_k} \big(\widetilde{\Omega}(x_\alpha)-\widetilde{\Omega}(y_\alpha)\big)\Big\|_2 = \log(2) rs \sqrt{2^k}.$$
\end{lemma}

\begin{proof}
We suggest the reader to take a pencil and scribble some trees and rectangles.
It is clear from (\ref{eq:chunk}) that if $x$ has a single nonzero chunk, then $\widetilde\Omega(x)=0$, so certainly
$\widetilde\Omega(x_\alpha)=0$ if $|\alpha|\geq k-s=r$. On the other hand, if $|\alpha|<r$, then $x_\alpha$ is the sum of $2^{r-|\alpha|}$ many chunks of equal norm and so ${\widetilde{\Omega}}(x_\alpha)=(r-|\alpha|)\log(2) x_\alpha$. Regarding $y_\alpha$, for  $|\alpha| \geq k-r=s$, it is a sum of unit vectors in $2^{k-|\alpha|}$ disjoint $A_i$'s, so
${\widetilde{\Omega}}(y_\alpha)=(k-|\alpha|)\log(2)\, y_\alpha$
and if $|\alpha|<s$ it is the sum of equal norm vectors in all the $2^{r}$ sets $A_i$'s, so
${\widetilde{\Omega}}(y_\alpha)=r \log(2) y_\alpha$.

Finally we compute
\begin{align*}
\frac{\sum_{\alpha \in \mathscr T_k} \big({\widetilde{\Omega}}(x_\alpha)-{\widetilde{\Omega}}(y_\alpha)\big)}{\log 2}&=
\sum_{|\alpha|<r} \widetilde{\Omega}(x_\alpha)
-
\sum_{|\alpha|\geq s} \widetilde{\Omega}(y_\alpha)
-
\sum_{|\alpha|< s} \widetilde{\Omega}(y_\alpha)\\
&=
\sum_{|\alpha|<r} (r-|\alpha|) x_\alpha
-
\sum_{|\alpha|\geq s}(k-|\alpha|) y_\alpha
-
\sum_{|\alpha|< s} ry_\alpha,
\end{align*}
which, taking into account that for each $0\leq n\leq k$
$$
\sum_{|\alpha|=n} x_\alpha = \sum_{|\alpha|=n} y_\alpha= x_\varnothing= y_\varnothing= \sum_{i\in 2^r\times 2^s} e_i,
$$
is equal to
$$\Big(-r(s-1) -\sum_{n \leq r}n +\sum_{n<r}n\Big) x_\varnothing=-rs x_\varnothing$$
and its norm is $rs \sqrt{2^k}$. \end{proof}

\begin{theorem}\label{th:main} Let $(A_i)$ be a partition of $\,\N$.
Let $\,\Omega:\ell_2^0\To \ell_2$ be the Kalton-Peck map and let $\widetilde \Omega$ be the centralizer associated to $(A_i)$, as in Section~\ref{sec:chunk}. If
for every $n$ there exist $n$ sets in $(A_i)$ whose cardinality is at least $n$,
then
 $\widetilde \Omega\,\Omega \nsim 0$.
 In particular,  $\Ext^2(\ell_2, \ell_2)\neq 0$.
\end{theorem}

\begin{proof} For each  $k$ we can find a subspace generated by $4^k$ vectors of the basis and where $\widetilde{\Omega}$ identifies with the centralizer $\widetilde\Omega_{2^k, 2^k}$ as described in Lemma \ref{lem:rs} for $r=s=k$. Then the two previous Lemma together with the symmetry of Kalton-Peck map yield that if $\widetilde{\Omega }	, \Omega \sim 0$ and $K$ is the associated constant then one gets the contradiction
$$\log(2)k^2 2^k= \Big\| \sum_{\alpha \in \mathscr T_k} \big(\widetilde{\Omega}(x_\alpha)-\widetilde{\Omega}(y_\alpha)\big)\Big\|_2  \leq 2\sqrt{2} K k2^k.\vspace{-20pt}$$
\end{proof}

Since the extension induced by $\widetilde \Omega\,\Omega$ lives in $\mathsf{B}$ and is not trivial in  $\mathsf{Q}$ one gets

\begin{cor}
$\Ext^2_\mathsf B(\ell_2, \ell_2)\neq 0$.
\end{cor}

\section{Miscellaneous applications}\label{sec:appl}

The remainder of the paper is devoted to presenting a number of applications of the main result. These range from classical operator theory to Banach modules, including some issues about the difference between $\Ext^2$ and $\Ext^2_\mathsf B$.

\subsection{Spinning around Hilbert space}

Hilbert space is a central object of Banach space theory in many respects. The ensuing application exploits this fact in a rather direct way. 
A Banach space $X$ contains $\ell_2^n$ uniformly complemented if there is a constant $C$ such that, for every $n\in\N$ there are operators $I:\ell_2^n\To X$ and $P:X\To \ell_2^n$ such that $PI$ is the identity on $\ell_2^n$ and $\|I\|\|P\|\leq C$. This property is ``self-dual'' ($X$ has it exactly when $X^*$ does) and weaker than $B$-convexity (= nontrivial type $p>1$).

\begin{cor}
If $X$ and $Y$ are Banach spaces containing $\ell_2^n$ uniformly complemented, then $\Ext^2_\mathsf B(X,Y)\neq 0$.
\end{cor}

\begin{proof}
We write the proof assuming $X$ and $Y$ separable. The general case does not present any additional difficulty.
Given a (separable) Banach space $X$ we fix an isometric quotient map $Q_1:\ell_1\To X$ and we set $K^1(X)=\ker Q_1$. Then we fix a quotient $Q_2: \ell_1\To K^1(X)$ and set $K^2(X)=\ker Q_2$.
Splicing through $K^1(X)$ we obtains the exact sequence\vspace{-5pt}
\begin{equation*}
\xymatrixrowsep{0.5pc}
\xymatrixcolsep{2.5pc}
\xymatrix{
0\ar[r] & K^2(X)\ar[r] & \ell_1  \ar[rr]  \ar[dr]_<<<<<{Q_2}  &&\ell_1 \ar[r]^{Q_1} & X\ar[r]& 0\\
& && K^1(X) \ar[ur]}
\end{equation*}
that allows us to ``reduce'' the length of extensions. Indeed, for all Banach spaces $Y$, one has
$\Ext^2_\mathsf B(X,Y)=\Ext_\mathsf{B}(K^1(X),Y)$. This follows from the fact that $\ell_1$ is projective in the category of Banach spaces. If we specialize to the case where $X=H$ is a separable Hilbert space, we get
$\Ext^2_\mathsf{B}(H,Y)= \Ext_\mathsf{B}(K^1(H),Y)$ for all $Y$. Now, by the main result $\Ext^2_\mathsf{B}(H,H)= \Ext_\mathsf{B}(K^1(H),H)$ is nonzero and a uniform boundedness argument (cf. \cite[Theorem 2]{2c-2004}) yields
$\Ext_\mathsf{B}(K^1(H),Y)\neq 0$. We have thus arrived at:
\medskip

{\sc Claim.} If $Y$ contains $\ell_2^n$ uniformly complemented, then $\Ext^2_\mathsf{B}(H,Y)\neq 0$. In particular, $\Ext^2_\mathsf{B}(H,X^*)\neq 0$.\medskip

In view of the duality formula $\Ext^2_\mathsf{B}(E,F^*)=\Ext^2_\mathsf{B}(F, E^*)$, valid for all Banach spaces $E$ and $F$, one also has $\Ext^2_\mathsf{B}(X,H^*)\neq 0$. But since  $H^*$ is isometric to $H$ this means that $\Ext_\mathsf{B}(K^1(X),H)\neq 0$ and applying again
\cite[Theorem 2]{2c-2004} we obtain $\Ext^2_\mathsf{B}(X,Y)=\Ext_\mathsf{B}(K^1(X),Y)\neq 0$.
\end{proof}

\subsection{Impossibility of extending certain operators} Kalton shows in \cite{kaltCK} that
the Kalton-Peck spaces $Z_p(\varphi)$ for $1<p<\infty$ have the remarkable property that every $C(K)$-operator defined on any subspace of $Z_p(\varphi)$ extends to  an operator $Z_p(\varphi) \To C(K)$. Curiously, one has:

\begin{cor} Let $X$ be a separable Banach space containing $\ell_2^n$ uniformly complemented. There exists an embedding $u: X \To C[0,1]$ and an operator $v: \ell_2\To C[0,1]/u[X]$ that cannot be extended to $Z_2$.\end{cor}

\begin{proof} We will make the proof for $X=\ell_2$ and leave the reader to derive from that and the Claim above the general case. Let  $u: \ell_2 \To\ell_\infty$ be an isomorphic embedding. Since $\ell_\infty$ is an injective Banach space one has a (pull-back) commutative diagram
$$
\xymatrix{
0\ar[r] & \ell_2 \ar[r]^u & \ell_\infty  \ar[r]   &\ell_\infty/u[\ell_2]  \ar[r] & 0\\
0\ar[r] & \ell_2 \ar[r]\ar@{=}[u] & \widetilde Z_2 \ar[r]\ar[u]_V   &\ell_2 \ar[u]_v \ar[r] & 0 }
$$
where $V$ is an extension of $u$ and $v$ is the factorization of $V$ through the quotient map.
The (self-adjoint, if $\mathbb K=\mathbb C$) subalgebra of $\ell_\infty$ generated by $V[\widetilde Z_2]$ is a Banach space isomorphic to $C[0,1]$, which generates a ``separable'' commutative diagram
$$
\xymatrix{
0\ar[r] & \ell_2 \ar[r]^<<<<u & C[0,1] \ar[r]   &C[0,1]/u[\ell_2]  \ar[r] & 0\\
0\ar[r] & \ell_2 \ar[r]\ar@{=}[u] & \widetilde Z_2 \ar[r]\ar[u]_V   &\ell_2 \ar[u]_v \ar[r] & 0 }
$$
Let $\PO$ be the push-out of $v$ and the inclusion of $\ell_2$ into $Z_2$, the original Kalton-Peck space.
Since $\widetilde\Omega\,\Omega \nsim 0$, a look at the commutative diagram
\begin{equation*}
\xymatrixrowsep{0.75pc}\xymatrix{0\ar[r] & \ell_2 \ar[r]^<<<<u   & C[0,1] \ar[rr]   \ar[dr]  &&\PO \ar[r]  &\ell_2\ar[r] & 0 \\
& && C[0,1]/u[\ell_2] \ar[ur]   \\
0\ar[r] & \ell_2\ar[r]  \ar@{=}[uu] & \widetilde Z_2 \ar[rr]  \ar[uu] \ar[dr]  &&Z_2 \ar[r]  \ar[uu]  &\ell_2\ar[r]  \ar@{=}[uu]& 0 \\
& && \ell_2 \ar[ur]  \ar[uu]^<<<<v }
\end{equation*}
reveals that the upper-right push-out sequence cannot split, and thus $v$ cannot be extended to $Z_2$.\end{proof}

A simple consequence of the previous argument and the standard reduction, which for separable Banach spaces has the form $\Ext^2_{\mathsf B}(X,Y) = \Ext_{\mathsf B}(X, \ell_\infty/Y) = \Ext_{\mathsf B}(K^1(X), Y)$, is that the following spaces of extensions are nonzero:
$$\Ext(\ell_2, \ell_\infty/\ell_2), \qquad \Ext(\ell_2, C[0,1]/\ell_2),\qquad\Ext_{\mathsf B}(K^1(\ell_2), \ell_2).$$
All these results hold replacing everywhere $2$ by any $1<p<\infty$. The following one is however specific of $2$: By the main result of \cite{castrica}, we can conclude the existence of a bounded bilinear form on $K^1(\ell_2)$ that cannot be extended to $\ell_1$.


\subsection{Commuting centralizers}\label{sec:comm}
This Section is, to a large extent, independent on the rest of the paper. It contains some remarks about the nature of the ``splicing mapping'' (=Yoneda product)
\begin{equation}\label{eq:ExE-->E2}
\Ext(\ell_p,\ell_p)\times \Ext(\ell_p,\ell_p)\To \Ext^2(\ell_p,\ell_p).
\end{equation}
The first thing one must know about this mapping is the following.

\begin{proposition}\label{prop:FF=0}
If $\,\Phi:\ell_p^0\To\ell_p$ is a centralizer, then $\Phi\,\Phi\sim 0$.
\end{proposition}

This rather surprising fact was first observed (for centralizers on K\"othe spaces with nontrivial type $p>1$) in \cite[Section 6.4]{cck}, a paper that makes heavy use of interpolation theory.
 Given our commitment to avoid interpolation, let us sketch a more direct proof based on the lifting part of Lemma~\ref{lem:crit-operative}, Lemma~\ref{lem:compa} and an estimate hidden in \cite{tensor} which did not found any application until now.

\begin{proof}
In view of Lemma~\ref{lem:compa} we can assume that $\Phi(f)=u|f|^{1/2}\Phi_{2p}\big(|f|^{1/2}\big)$ for some centralizer $\Phi_{2p}:\ell_{2p}^0\To\ell_{2p}^0$, where $f=u|f|$ is the polar decomposition. It is proven in \cite[Last observation on p. 335]{tensor} that the map
$$
\Lambda(f)=\Big( \frac{u\big(\Phi_{2p}\big(|f|^{1/2}\big)\big)^2}{2}, \Phi(f) \Big)
$$
is a centralizer from $\ell_p^0$ to $\ell_p\oplus_{\Phi}\ell_p^0$  which obviously lifts $\Phi$.
\end{proof}

Given quasilinear maps $\Phi_i: \ell_p^0\To\ell_p$ for $1\leq i\leq 4$, we write $\Phi_1\,\Phi_2\,\sim\,\Phi_3\,\Phi_4$,  if the exact sequences associated to $\Phi_1\,\Phi_2$ and $\Phi_3\,\Phi_4$ are equivalent in $\Ext^2(\ell_p,\ell_p)$. We say that $\Phi$ and $\Psi$ commute
if $\Phi\,\Psi\,\sim\,\Psi\,\Phi$. One has:

\begin{cor}\label{cor:commuting} Let $\Phi, \Psi: \ell_p^0\To\ell_p$ be centralizers.
\begin{itemize}
\item[(1)] $\Phi\,\Psi \sim -\Psi\,\Phi$.
\item[(2)] $\Phi\,\Psi \sim 0$ if and only if $\Psi\,\Phi \sim 0$ if and only if  $\Phi$ and $\Psi$ commute
\item[(3)] $\Phi\,\Psi \sim 0$ if and only if  $(\Phi+\Psi) (\Phi-\Psi) \sim 0$
\end{itemize}
\end{cor}
\begin{proof}
The mapping in (\ref{eq:ExE-->E2}) is bilinear with respect to the natural linear structures and the linear structure of $\Ext$ corresponds to the pointwise operations of quasilinear maps. Therefore, with a slight abuse of notation and the obvious meaning, we have, in view of Proposition~\ref{prop:FF=0},
$$
0\sim (\Phi+\Psi) (\Phi+\Psi)\sim \Phi\,\Phi + \Phi\,\Psi + \Psi\,\Phi + \Psi\,\Psi\sim
\Phi\,\Psi + \Psi\,\Phi.
$$
This proves the first item and the rest is straightforward.
\end{proof}

Recalling the extension part of Lemma~\ref{lem:crit-operative} we see that (2) implies  that $\Phi:\ell_p^0\To\ell_p^0$ has a quasilinear extension to $\ell_p^0\oplus_{\Psi}\ell_p^0$ if and only if $\Psi:\ell_p^0\To\ell_p^0$ has a quasilinear extension to $\ell_p^0\oplus_{\Phi}\ell_p^0$.
This somewhat inexplicable fact can be explained, at least for $p=2$, as follows. The dual of $\ell_2$ can be identified with $\ell_2$ itself through the pairing $\langle x, y\rangle= \sum_nx(n)y(n)$. Thus, every centralizer $\Upsilon$ on $\ell_2$ has a dual centralizer $\Upsilon^*$ that corresponds to the dual sequence.
It follows from \cite[Section 3.3 and Corollary~4]{cabe} that one can
take $\Upsilon^*=-\Upsilon$. In particular, the dual sequence of the four-term sequence associated to $\Phi\,\Psi$ is that associated to $\Psi\,\Phi$. So,  $\Phi\,\Psi \sim 0$ if and only if $\Psi\,\Phi \sim 0$.


\subsection{The Enflo-Lindenstrauss-Pisier map.}
We do not know if Proposition~\ref{prop:FF=0} is true for arbitrary quasilinear maps instead of centralizers.
In this regard it is remarkable that it is so for the first quasilinear map that appeared in Banach space theory, namely the map
constructed by Enflo, Lindenstrauss and Pisier in their solution of the ``three-space problem'' (cf. \cite[Section 4]{elp}).

The seed of that mapping is the function $\Gamma_1 : \ell_2^3 \To \ell_2^3$ given by
\[
\Gamma_1(x, y, z) = \Big(x, y, \frac{x \left|y\right|}{(\left|x\right|^2 + \left|y\right|^2)^{1/2}}\Big).
\]
Let $P_1$ be the projection of $\ell_2^3$ onto the first two coordinates. Notice that $\Gamma_1(x, y, z) = \Gamma_1(P_1(x, y, z))$. Now we inductively define $ \Gamma_n, P_n: \ell_2^{3^n}\To \ell_2^{3^n}$ by the formul\ae
\begin{align*}
\Gamma_n(x, y, z) &= \Big(\Gamma_{n-1}(P_{n-1}(x)), \Gamma_{n-1}(P_{n-1}(y)), \frac{P_{n-1}(x) \|P_{n-1}(y)\|}{(\|P_{n-1}(x)\|^2 + \|P_{n-1}(y)\|^2)^{\frac{1}{2}}}\Big)\\
P_n(x, y, z) &= (P_{n-1}(x), P_{n-1}(y), 0)
\end{align*}
Then define $\Gamma : \ell_2^0 = \ell_2^0(\ell_2^{3^n}) \longrightarrow \ell_2 = \ell_2(\ell_2^{3^n})$ by $\Gamma((x_n)_n) = (\Gamma_n(x_n))_n$.

\begin{proposition} $\Gamma\,\Gamma\sim 0$ in $\Ext^2(\ell_2,\ell_2)$.
\end{proposition}
\begin{proof}
We claim that
$
\Gamma_n(\Gamma_n(u+v) - \Gamma_n(u) - \Gamma_n(v)) = 0
$ 
 for every $u,v \in \ell_2^{3^n}$ and all $n$. This implies that  condition (b) in Lemma~\ref{lem:H} is satisfied with $H=0$.

We proceed by induction.
Write $u=(x_1, y_1, z_1), v=(x_2, y_2, z_2)$.
Then, when $n=1$ one has
\[
\Gamma_1(x_1 + x_2, y_1 + y_2, z_1 + z_2) - \Gamma_1(x_1, y_1, z_1) - \Gamma_1(x_2, y_2, z_2) = \Big(0, 0, \frac{(x_1 + x_2)\left|y_1 + y_2\right|}{(\left|x_1 + x_2\right|^2 + \left|y_1 + y_2\right|^2)^{\frac{1}{2}}}\Big)
\]
so it follows that $\Gamma_1(\Gamma_1(u + v) - \Gamma_1(u) - \Gamma_1(v)) = 0$.

Assume the claim true for $n-1$. Then
\begin{align*}
\Gamma_{n}(x_1 + x_2, y_1 &+ y_2, z_1 + z_2) - \Gamma_{n}(x_1, y_1, z_1) - \Gamma_{n}(x_2, y_2, z_2) \\
&\leq \left(0, 0, \frac{P_{n-1}(x_1 + x_2) \|P_{n-1}(y_1 + y_2)\|}{(\|P_{n-1}(x_1 + x_2)\|^2 + \|P_{n-1}(y_1 + y_2)\|^2)^{\frac{1}{2}}}\right)
\end{align*}
and the claim follows as  we have $\Gamma_n(0,0,z)=0$ for all $z$.
\end{proof}

\subsection{Incomparable centralizers}
The study of the ``order'' structure of $\Ext(X,Y)$ spaces is still incipient.
Suppose we are given an exact sequence
\begin{equation}\label{eq:weare}
\xymatrix{
0\ar[r] & Y\ar[r] &Z \ar[r] & X\ar[r]& 0 }
\end{equation}
We say that
$0\To Y'\To Z'\To X\To 0$ is a push-out of (\ref{eq:weare}) if one has a commutative diagram
$$
\xymatrix{
0\ar[r] & Y\ar[r]\ar[d]_u &Z \ar[r] \ar[d] & X\ar[r] \ar@{=}[d] & 0\\
0\ar[r] & Y'\ar[r] &Z' \ar[r] & X\ar[r]& 0 }
$$
Dually, we say that $0\To Y\To Z''\To X''\To 0$ is a pull-back of (\ref{eq:weare}) if one has a commutative diagram
$$
\xymatrix{
0\ar[r] & Y\ar[r]  \ar@{=}[d] &Z \ar[r] \ar[d] & X\ar[r]  & 0\\
0\ar[r] & Y\ar[r] &Z'' \ar[r] \ar[u] & X'' \ar[u]^v \ar[r]& 0 }
$$
In \cite[Theorem 2.1]{ccky} it was shown that for every separable Banach space $X$ the space $\Ext_{\mathsf B}(X, C[0,1])$ admits an ``initial'' element $0\To C[0,1] \To Z \To X\To 0$
in the sense that any other element of $\Ext_{\mathsf B}(X, C[0,1])$ arises as a push-out for a suitable endomorphism $u$ of $C[0,1]$. In \cite{ccfm} it was shown that $\Ext(\ell_2, \ell_2)$ contains no initial (or final, via the obvious pull-back definition) element, thus answering a question in \cite{more}. The results in this paper show that the ``order'' structure of
$\Ext(\ell_2, \ell_2)$ is rather involved. With the same notations as in the main Theorem:

\begin{cor} The short exact sequences induced by $\widetilde\Omega$ and $\Omega$ are incomparable, that is, neither of them is a pull-back or a push-out of the other.
\end{cor}
\begin{proof} Let  $\Psi$ and $\Phi$ be centralizers on $\ell_2$ and let us show that if (the extension induced by) $\Psi$ is either a push-out or a pull-back of (that induced by) $\Phi$, then $\Psi\,\Phi\sim 0$. Assume that $\Psi$ is a push-out of $\Phi$ and write $\Psi\sim u\circ \Phi$, where $u$ is an operator on $\ell_2$, with the obvious and harmless abuse of notation. Then
$$
\Psi\,\Phi\sim u\circ(\Phi\,\Phi)\sim u\circ 0\sim 0.
$$
If we assume that $\Psi\sim \Phi\circ u$ instead, then
$$
\Phi\,\Psi \sim \Phi\, (\Phi\circ u) \sim (\Phi\,\Phi)\circ u\sim 0\qquad\implies\qquad\Psi\,\Phi\sim 0.\vspace{-20pt}
$$
\end{proof}

\subsection{On $\Ext^2(\ell_p,\ell_p)$ for $0<p<\infty$}\label{sec:all p} The purpose of this section is to prove that the reflections of the centralizers used in Theorem~\ref{th:main} provide nontrivial elements of $\Ext^2(\ell_p,\ell_p)$ for each $1\leq p<\infty$ and a nontrivial element of $\Ext^2_{\ell_\infty}(\ell_p,\ell_p)$ for $0<p<\infty$.

\begin{lemma}
Let $X, Y, W$ be sequence spaces, $\Psi: X^0\To W$ and $\Phi: W^0\To Y$ support preserving centralizers. Assume further that $Y$ is the dual of some sequence space.
Then $\Phi\,\Psi\sim 0$ in $\Ext^2(X, Y)$ if and only if the spliced sequence
\begin{equation}
\xymatrixrowsep{0.5pc}
\xymatrix{
0\ar[r] & Y \ar[r] & Z(\Phi)  \ar[rr]  \ar[dr]  && Z(\Psi)   \ar[r] & X\ar[r]& 0 \\
& && W \ar[ur]  \\}
\end{equation}
is trivial in $\Ext^2_{\ell_\infty}(X,Y)$.
\end{lemma}

\begin{proof}
This is clearly a manifestation of the amenability of the Banach algebra $\ell_\infty$.
Regrettably, the proof is not so direct as it should be. It suffices to see that if $\Phi$ can be extended to a  quasilinear map $\Gamma:  W^0\oplus_\Psi X^0 \To Y$, then it can be extended to a  quasilinear centralizer as well. We write the proof for complex scalars. The real case then follows by the usual change of base procedure. The advantage of having complex scalars is that every $a:\N\To\mathbb C$ with $\|a\|_\infty\leq 1$ can be written as the average of four unitaries:
$$
a=\frac{u_1+u_2+u_3+u_4}{4},
$$
and so every quasilinear map commuting with unitaries is a centralizer. We will not insist on this point. Let us assume that $\Phi$ commutes with the unitary group in the sense that $\Phi(uw)=u\Phi(w)$ for every finitely supported $w\in W$ and every $u\in U$, the unitary group of $\ell_\infty$. The general case follows immediately.

Let $m$ be an invariant mean for $\ell_\infty(U)$. Such an $m$ exists because $U$ is abelian, hence amenable. If you find this paragraph enigmatic, take a look on any book about invariant means, for instance, Greenleaf's \cite{greenleaf}. Let, further, $Y_*$ be the predual of $Y$.
Every bounded function $f:U\To Y$ has a weak* integral defined as
$$
\left\langle\int_U f(u)\,dm(u), y_*\right\rangle=
\int_U  \left\langle f(u), y_*\right\rangle dm(u)\qquad(y_*\in Y_*).
$$
We define a mapping $\Lambda:  W^0\oplus_\Psi X^0 \To Y$ by the weak* integral
\begin{equation}
\Lambda(w,x)=\int_Uu^{-1}\Gamma(u(w,x))\,dm(u).
\end{equation}
 The definition is correct because for finitely supported $z$ and $x$ the orbit $U(z,x)=\{(uz,uz): u\in U\}$ is confined in some finite dimensional subspace of $Z(\Psi)$ (actually of $W^0\oplus_\Psi X^0$),  where $\Gamma$ is bounded. Let us see that $\Lambda$ is a centralizer extending $\Phi$. First of all since $\Gamma$ is an extension of $\Phi$, we have
$$
\Lambda(w,0)=\int_Uu^{-1}\Gamma(u(w,0))\,dm(u)=
\int_Uu^{-1}\Phi(uw)\,dm(u)= \int_U\Phi(w)\,dm(u)= \Phi(w),
$$
so $\Lambda$ extends $\Phi$. Clearly, $\Lambda$ is homogeneous.
To check quasilinearity, let $Q$ be the quasilinearity constant of $\Gamma$ and pick points $(w,x)$ and $(w',x')$. We have
\begin{align*}
\|\Lambda\big((w,x)&+(w',x')\big)- \Lambda(w,x)-  \Lambda(w',x')\|
\\
&=
\left\|	 \int_Uu^{-1}\big(\Gamma(u((w,x)+(w',x')))- \Gamma(u(w,x))-  \Gamma(u(w',x'))\big)\,dm(u)   \right\|
\\
&\leq
	 \int_U \left\| \big(\Gamma(u((w,x)+(w',x')))- \Gamma(u(w,x))-  \Gamma(u(w',x'))\big) \right\|\, dm(u)
	 \\
&\leq
	 \int_U Q\big{(}\|u(w,x)\| + \|u(w',x')\| \big{)} \, dm(u)
	 \\
	 &\leq M\big{(}\|(w,x)\| + \|(w',x')\| \big{)},
\end{align*}
where $M$ depends only on $Q$ and the centralizer constant of $\Psi$. Finally, $\Lambda$ is a centralizer since it commutes with the actions of $U$ (by the invariance of $m$) and it is quasilinear.
\end{proof}

The requirement that $Y$ is a dual sequence space can be considerably relaxed: the Lemma is true when $X$ and $W$ are sequence spaces and $Y$ is any Banach $\ell_\infty$-module. The proof, however, would lead us too far away from our present subject.

\begin{proposition}\label{prop:trivialornot}
Let\, $0<q<p<\infty$. Let $\Phi_p, \Psi_p:\ell_p^0\To\ell_p$ be support preserving centralizers  and let\, $\Phi_q,\Psi_q$ be the corresponding reflections on $\ell_q$. Then:
\begin{itemize}
\item $\Phi_p\, \Psi_p$ induces a trivial sequence of quasi Banach modules if and only if $\Phi_q\, \Psi_q$ does.
\item If $q\geq 1$, then $\Phi_p\, \Psi_p$ induces a trivial sequence of Banach spaces if and only if $\Phi_q	\, \Psi_q$ induces a trivial sequence of quasi-Banach spaces.
\end{itemize}
\end{proposition}

\begin{proof}
By the preceding Lemma, it suffices to prove the first item. Please note that,
replacing every occurrence of ``trivial'' by ``nontrivial'' does not alter the Proposition.
Also, if $q>1$, then the word ``quasi'' can be eliminated everywhere in the Proposition since $\Phi_q\, \Psi_q$ induces a sequence of Banach spaces, after renorming.
\medskip

We first prove that for $p>q$, ``triviality at $p$'' implies ``triviality at $q$'', exploiting the lifting part of Lemma~\ref{lem:crit-operative}. The hypothesis means that there is a  centralizer
$\Gamma: \ell_p^0 \To \ell_p^0\oplus_{\Phi_p}\ell_p^0$ such that that $\pi(\Gamma(f))=\Psi_p(f)$ for every $f\in\ell_p^0$. Let us define $\Gamma_q: \ell_q^0 \To \ell_q^0\oplus_{\Phi_q}\ell_q^0$ taking
$$
\Gamma_q(f)= u|f|^{q/r}\Gamma \big{(}|f|^{q/p}\big{)},
$$
where $f=u|f|$ is the polar decomposition. Obviously $\Gamma_q$ is a lifting of $\Psi_q$ since
$$
\pi(\Gamma_q(f))= \pi \big{(}u|f|^{q/r}\Gamma \big{(}|f|^{q/p}\big{)} \big{)}
=
u|f|^{q/r}  \pi \big{(} \Gamma \big{(}|f|^{q/p}\big{)} \big{)}
=
u|f|^{q/r}  \Psi_p  \big{(}|f|^{q/p}\big{)} = \Psi_q(f).
$$
To check that $\Gamma_q$ is a centralizer, pick $a\in\ell_\infty$ and $f\in\ell_q^0$ with polar decompositions $a=v|a|$ and $f=u|f|$, respectively. Then
$$
\Gamma_q(af)= vu |af|^{q/r}\Gamma( \,|af|^{q/p} \,),\qquad
a\Gamma_q( f)= a  u |f|^{q/r}\Gamma( \,|f|^{q/p} \,)
$$
and the centralizer property of $\Gamma_q$ follows from H\"older inequality and the corresponding estimate of $\Gamma$. Quasilinearity now follows from Lemma~\ref{lem:autom}.\medskip

We finally prove that ``triviality at $q$'' implies ``triviality at $p$''.
The key point is that if $\Phi_p$ and $\Phi_q$ are as in Lemma~\ref{lem:compa}, then
$\Hom(\ell_r,   Z(\Phi_q))=  Z(\Phi_p)$. In our ``finitistic'' setting, the preceding identity should be interpreted as follows:
Every $(g,f)\in  \ell_p\oplus_{\Phi_p}\ell_p^0$ induces a homomorphism from $\ell_r$ to $Z(\Phi_q)$ by the rule
$$h\in\ell_r\longmapsto (gh,fh)\in Z(\Phi_q)$$
and the quasinorm of $(g,f)$ in $\ell_p\oplus_{\Phi_p}\ell_p^0$ is equivalent to the quasinorm of the operator defined two lines ago. This defines an operator (actually a homomorphism) $\ell_p\oplus_{\Phi_p}\ell_p^0\To \Hom(\ell_r,   Z(\Phi_q))$ with dense range whose extension to $Z(\Phi_p)$ is the required isomorphism. Thus, keeping the same notations, we have a commutative diagram
\begin{equation}\label{dia:iden}
\xymatrix{
0\ar[r] & \ell_p  \ar@{=}[d]\ar[r] & Z(\Phi_p)  \ar[d] \ar[r] & \ar@{=}[d] \ell_p \ar[r]&  0\\
0\ar[r] & \Hom(\ell_r, \ell_q) \ar[r] & \Hom(\ell_r, Z(\Phi_q) ) \ar[r] & \Hom(\ell_r, \ell_q) \ar[r]& 0
}
\end{equation}
This can be seen in \cite[Corollary~2]{cabe}, where the definition of a centralizer is slightly different from ours.

That said, assume that $\Phi_q\,\Psi_q$ induces a trivial sequence of quasi Banach modules. Go back to the basic criterion and write down the ``splitting'' diagram of modules and homomorphisms
\begin{align}\label{dia:splitq}
\xymatrixrowsep{1.5pc}
\xymatrix{ & 0 \ar[d]  & 0\ar[d]\\
 & \ell_q   \ar@{=}[r]\ar[d] & \ell_q \ar[d] \\
0\ar[r] &  Z(\Phi_q)  \ar[d] \ar[r] & \square \ar[d] \ar[r] & \ell_q \ar[r] \ar@{=}[d]&  0\\
0\ar[r] & \ell_q \ar[r]\ar[d] &  Z(\Psi_q)   \ar[r]\ar[d] & \ell_q \ar[r]& 0\\
& 0 & 0
}
\end{align}
Applying $\Hom(\ell_r,-)$ to the above diagram and using the identifications provided by (\ref{dia:iden}) we obtain another commutative diagram of quasi Banach modules and homomorphisms
\begin{equation*}
\xymatrix{ & 0 \ar[d] & 0\ar[d] \\
 & \ell_p   \ar@{=}[r]\ar[d] & \ell_p \ar[d] \\
0\ar[r] &   Z(\Phi_p) \ar[d] \ar[r] & \Hom(\ell_r,\square) \ar[d] \ar[r] & \ell_p \ar[r] \ar@{=}[d]&  0\\
0\ar[r] & \ell_p \ar[r]\ar[d] &   Z(\Psi_p)  \ar[r]\ar[d] & \ell_p \ar[r]& 0\\
& 0 & 0
}
\end{equation*}
If we want to use this diagram to conclude that the spliced sequence $\Phi_p\, \Psi_p$ is trivial in $\Ext^2_{\ell_\infty}(\ell_p,\ell_p)$ we have to show that the two short sequences passing through $\Hom(\ell_r,\square)$ are exact. This amounts to checking the surjectivity of the two arrows starting at $\Hom(\ell_r,\square)$.

Let us begin with the vertical one. Having a look at Diagram~\ref{dia:splitq} we realize that the map $\Hom(\ell_r,\square)\To Z(\Psi_p)=\Hom(\ell_r,Z(\Psi_q))$ is onto if and only if every homomorphism $u:\ell_r\To Z(\Psi_q)$ lifts to $\square$. And this is indeed the case because otherwise the pull-back diagram
$$
\xymatrix{
0 \ar[r] & \ell_q \ar[r]  \ar@{=}[d] & \square\ar[r]   & Z(\Psi_q) \ar[r] & 0\\
0 \ar[r] & \ell_q \ar[r]  & \PB \ar[u]\ar[r]   & \ell_r \ar[r] \ar[u]^u & 0
}
$$
would provide a nontrivial element of $\Ext_{\ell_\infty}(\ell_r,\ell_q)$. However, the main result of \cite{cabe} shows that $\Ext_{\ell_\infty}(\ell_r,\ell_q)=0$ unless $r=q$. In a similar vein, if the map $\Hom(\ell_r,\square)\To  \ell_p =\Hom(\ell_r, \ell_q)$ fails to be onto, one can produce a nontrivial extension of modules
$$
\xymatrix{
0 \ar[r] & Z(\Phi_q) \ar[r]  & \PB \ar[r]   & \ell_r \ar[r]  & 0
}
$$
using pull-back. Which cannot be since $\Ext_{\ell_\infty}(\ell_r,\ell_q)=0$ implies that $\Ext_{\ell_\infty}(\ell_r, Z)=0$ for every quasi Banach module fitting in a short exact sequence of modules $0\To \ell_q \To Z\To \ell_q\To 0$, in particular for $Z= Z(\Phi_q)$.\end{proof}


\subsection{$\Ext^2(\ell_1,\mathbb K)$ and the role of the ambient category}
We close the Section with an example showing that $\Ext^2$ heavily depends on the ambient category. Note that $\Ext^2_\mathsf B(\ell_1, Y)= 0$ and $\Ext^2_\mathsf{B}(X,\mathbb K)= 0$ for all Banach spaces $X$ and $Y$. Indeed, one even has $\Ext_\mathsf B(\ell_1, Y)= 0$ by the lifting property of $\ell_1$ and $\Ext_\mathsf{B}(X,\mathbb K)= 0$ by the Hahn-Banach theorem.
However:

\begin{proposition}\label{prop:Ext2l1}
$\Ext^2(\ell_1,\mathbb K)\neq 0$.
\end{proposition}

The proof is based on the idea that, for most quasi-Banach modules $X$, minimal extensions of $X$ (defined at the end of Section~\ref{sec:crit}) are in correspondence with extensions of quasi Banach modules $0\To\ell_1\To Z\To X\To 0$. We will state and prove this just for the particular case that we need to carry out the proof of the Proposition.

Suppose that $X$ is an $\ell_\infty$-module. Let us say that $x\in X$
is finitely supported if there is a \emph{finite} $A\subset \N$ such that $x=1_A x$. The smallest $A$ for which the preceding identity holds is called the support of $x$ and will be denoted by $\supp(x)$. There is no conflict of interest with the usual meaning of support of a (finitely supported) sequence. A moment's reflection suffices to realize that:
\begin{itemize}
\item The finitely supported elements of a module $X$ constitute a (not generally closed and quite often dense) submodule that we denote by $X^0$.
\item Any  morphism, continuous or not, preserves finitely supported elements.
\item If $x$ has finite support, then the ``orbit'' $\ell_\infty\, x=\{ax:a\in\ell_\infty\}$ is a finite-dimensional subspace (actually submodule) of $X$.
\end{itemize}

\begin{lemma}\label{lem:phiPhi}
Let $X$ be a quasi-Banach module over $\ell_\infty$ and let $X^0$ be the submodule of finitely supported elements of $X$. Let $\phi: X^0\To\K$ be a quasilinear map. Then there is a quasilinear centralizer $\Phi: X^0\To \ell_1$ such that
\begin{equation}\label{eq:phi-Phi}
\big{|}\phi(x)-\langle\Phi(x), 1_\N\rangle\big{|}\leq C\|x\|
\end{equation}
for some $C$ and all $x\in X^0$. Any two centralizers having this property are strongly equivalent on $X^0$.
\end{lemma}

\begin{proof}[Proof of Lemma~\ref{lem:phiPhi}]
There is no loss of generality in we assume that $X$ is ``contractive'': $\|ax\|\leq \|a\|_\infty\|x\|$ for all $a\in\ell_\infty$ and $x\in X$ and that $\phi:X^0\To\K$ has $Q(\phi)\leq 1$.

Let us first dispose of the ``uniqueness'' part. Suppose that, for $i=1,2$, one has centralizers $\Phi_i:  X^0\To \ell_1$ such that
\begin{equation*}
\big{|}\phi(x)-\langle\Phi_i(x), 1_\N\rangle\big{|}\leq C_i\|x\|.
\end{equation*}
Note that
$$
\|\Phi_1(x)-\Phi_2(x)\|_1= \sup_{\|a\|_\infty\leq 1} \big{|}\langle\Phi_1(x)-\Phi_2(x), a\rangle\big{|}.
$$
But
\begin{align*}
&\langle\Phi_i(x), a\rangle=  \langle a \Phi_i(x), 1_\N\rangle, \\
& \big{|} \langle a \Phi_i(x), 1_\N\rangle - \langle  \Phi_i(ax), 1_\N\rangle  \big{|}\leq \|\Phi_i(ax)-a\Phi_i(x)\|_1\leq C(\Phi_i)\|a\|_\infty\|x\|,\\
& \big{|}\phi(ax)-\langle\Phi_i(ax), 1_\N\rangle\big{|}\leq C_i\|ax\|\leq C_i\|a\|_\infty \|x\|.
\end{align*}
Combining,
$$
\|\Phi_1(x)-\Phi_2(x)\|_1\leq \big{(} C_1+C_2+C(\Phi_1)+C(\Phi_2)\big{)}\|x\|.
$$
We pass to the construction of $\Phi$ which involves the fact that $c_0$ is a $K$-space: there is an absolute constant $C$ such that, for every quasilinear map $\varphi: c_0\To\K$ there is a linear map $\ell:c_0\To \K$ satisfying $|\varphi(a)-\ell(a)|\leq CQ(\phi)\|a\|_\infty$, where $Q(\varphi)$ is the quasilinearity constant of $\varphi$. 
This is a really deep result by Kalton and Roberts \cite{k-r}. They got $C\leq 200$ for $\K=\R$. It follows that $C\leq 200\sqrt{2}$ for $\K=\mathbb C$.

We consider the two-variable function $c_0\times X^0\To\K$ given by $(a,x)\longmapsto\phi(ax)$. Clearly, for each fixed $x\in X^0$, the function $a\in c_0\longmapsto\phi(ax)\in\K$ is quasilinear, with constant at most $\|x\|$ and so there is a linear map, depending on $x$, say $\ell_x:c_0\To \K$ such that
$$
|\phi(ax)-\ell_x(a)\|\leq C\|x\|\|a\|_\infty,\qquad(x\in X^0, a\in c_0).
$$
But, for fixed $x\in X^0$, the orbit $c_0x$ is finite dimensional and so $a\in c_0\longmapsto \phi(ax)\in\mathbb K$ is bounded. This implies that $\ell_x$ is a bounded functional on $c_0$ and therefore it is implemented by some $f_x\in\ell_1$ in the obvious way.

If we choose such an $f_x$ homogeneously, we obtain a homogeneous mapping $\Phi:X^0\To\ell_1$ such that
\begin{equation}\label{eq:homo}
|\phi(ax)-\langle\Phi(x), a\rangle\|\leq C\|x\|\|a\|_\infty,\qquad(x\in X^0, a\in c_0).
\end{equation}
Let us check that this $\Phi$ is the centralizer we are looking for.
\smallskip

$\bullet$ First of all, it is clear that (\ref{eq:homo}) implies (\ref{eq:phi-Phi}).
\smallskip

$\bullet$ $\Phi$ is quasilinear. Lemma~\ref{lem:autom} does not apply here. Since $\|f\|_1=\sup |\langle f, a\rangle|$, where $a$ runs on the unit ball of $c_0$, it suffices to check that if $x,y\in X^0$, then
$$
| \langle \Phi(x+y)-\Phi(x)-\Phi, a\rangle |\leq Q\|a\|_\infty\big{(} \|x\|+ \|y\|\big{)}
$$
which immediately follows from (\ref{eq:homo}) and the quasilinearity of $\phi$.
\smallskip

$\bullet$ $\Phi$ is a centralizer. It suffices to see that $\Phi$ obeys an estimate of the form $\|\Phi(bx)-b\Phi(x)\|_1\leq C\|b\|_\infty\|x\|$ for finitely supported $b$. But for each $a\in c_0$ we have
\begin{align*}
|\langle \Phi(bx), a\rangle- \phi(abx)| \leq C\|bx\|\|a\|_\infty\leq  C\|b\|_\infty\|x\|\|a\|_\infty.
\end{align*}
On the other hand, $\langle b\Phi(x), a\rangle= \langle \Phi(x), ab\rangle$, so
$$
|\langle b \Phi(x), a\rangle- \phi(abx)| \leq C\|abx\|\leq  C\|b\|_\infty\|x\|\|a\|_\infty,
$$
which completes the proof.
\end{proof}

\begin{proof}[Proof of Proposition~\ref{prop:Ext2l1}]
The main result and Proposition~\ref{prop:trivialornot}(b) imply that there are centralizers $\Phi, \Psi$ on $\ell_1$ for which the spliced sequence
\begin{equation}\label{dia:Ext2l1}
\xymatrixrowsep{0.5pc}
\xymatrix{
0\ar[r] & \ell_1 \ar[r] & Z(\Phi) \ar[rr]  \ar[dr]  && Z(\Psi)  \ar[r] & \ell_1\ar[r]& 0\\
& && \ell_1 \ar[ur]  }
\end{equation}
is not trivial in $\Ext^2(\ell_1, \ell_1)$. Let $S:\ell_1\To\K$ be the sum functional. We will prove that the push-out sequence
\begin{equation*}
\xymatrix{0\ar[r] & \ell_1 \ar[r]\ar[d]_ S & Z(\Phi)\ar[r]  \ar[d]  & Z(\Psi) \ar@{=}[d] \ar[r] & \ell_1 \ar@{=}[d]\ar[r]& 0 \\
0\ar[r] & \K \ar[r] & \PO \ar[r]    & Z(\Psi)  \ar[r] & \ell_1\ar[r]& 0 }
\end{equation*}
provides a nonzero element of $\Ext^2(\ell_1,\mathbb K)$. Note that the push-out space can we represented as $\PO=\mathbb K\oplus_\phi \ell_1$, where $\phi$ is a quasilinear extension of the function $x\in\ell_1^0\longmapsto\langle\Phi(x), 1_\N\rangle \in\K$, so the lower sequence in the preceding diagram is
\begin{equation*}
\xymatrixrowsep{0.5pc}
\xymatrix{
0\ar[r] & \K \ar[r] & \K\oplus_\phi \ell_1 \ar[rr]  \ar[dr]  && Z(\Psi)  \ar[r] & \ell_1\ar[r]& 0\\
& && \ell_1 \ar[ur]  }
\end{equation*}
Now, it suffices to see that the map $\phi: \ell_1^0\To\K$ cannot be extended to
$\ell_1^0\oplus_\Psi \ell_1^0= (\ell_1\oplus_\Psi \ell_1 )^0$ keeping quasilinearity.

Assume $\gamma$ is such an extension and let $\Gamma:  \ell_1^0\oplus_\Psi \ell_1^0\To \ell_1$ be the quasilinear centralizer provided by Lemma~\ref{lem:phiPhi}. Since $\gamma$ is an extension of $\phi$, the restriction of $\Gamma$ to $\ell_1^0$ is strongly equivalent to $\Phi$ and so the sequence in (\ref{dia:Ext2l1}) should be trivial in $\Ext^2(\ell_1,\ell_1)$, which is not the case.
\end{proof}

\end{document}